\documentclass[12pt]{article}

\usepackage{a4wide}
\usepackage{amsmath,amsfonts,amssymb}
\usepackage{bm}
\usepackage{color}
\usepackage[english]{babel}
\usepackage{graphicx}
\usepackage{todonotes}

\usepackage[plainpages=false]{hyperref}
\hypersetup{
		bookmarksopen=true,
		bookmarksnumbered=true,
		pdfborder={0 0 0},
		pdftitle={Finite Element Approximation of the Laplace-Beltrami Operator on a Surface with Boundary},
		pdfauthor={E. Burman, P. Hansbo, M.G. Larson, K. Larsson,  A. Massing},
		pdfcreator={E. Burman, P. Hansbo, M.G. Larson, K. Larsson,  A. Massing},
		colorlinks=true,
		linkcolor=blue,
		urlcolor=green	
	}

\numberwithin{equation}{section}

\newcommand{\bfV}{\boldsymbol{V}}
\newcommand{\bfI}{\boldsymbol{I}}

\newcommand{\tn}{|||}

\newcommand{\mcK}{\mathcal{K}}

\newcommand{\mcN}{\mathcal{N}}

\newcommand{\Gammah}{{\Gamma_h}}
\newcommand{\Gammahl}{{\Gamma_h^l}}
\newcommand{\nablas}{\nabla_\Gamma}
\newcommand{\nablash}{\nabla_{\Gamma_h}}
\newcommand{\divs}{\text{div}_\Gamma}
\newcommand{\Ps}{P_\Gamma}
\newcommand{\Psh}{P_{\Gammah}}
\newcommand{\Qs}{Q_\Gamma}
\newcommand{\Qsh}{Q_{\Gammah}}
\newcommand{\nuh}{\nu_{\partial \Gammahl}}
\newcommand{\nuhh}{\nu_{\partial \Gammah}}
\newcommand{\tth}{t_{\partial \Gammah}}
\newcommand{\tthl}{t_{\partial \Gammahl}}
\newcommand{\muh}{|B|}
\newcommand{\muhh}{|B_{\partial \Gammah}|}
\newcommand{\tp}{\widetilde{p}_{\partial \Gamma}}
\newcommand{\pb}{p_{\partial \Gamma}}

\newcommand{\gammah}{\gamma_\Gammah}
\newcommand{\gammahl}{\gamma_\Gammahl}
\newcommand{\Gammahlout}{\Gamma_{h,\text{out}}^l}

\newcommand{\IR}{\mathbb{R}}
\newtheorem{lem}{Lemma}[section]

\newtheorem{thm}{Theorem}[section]
\newtheorem{rem}{Remark}[section]

\newenvironment{proof}{\noindent \newline {\bf Proof.}}
{\hfill \mbox{\fbox{} } \newline}

\title{\bf Finite Element Approximation of the Laplace-Beltrami 
Operator on a Surface with Boundary}
\date{\today}
\author{Erik Burman\footnote{ Department of Mathematics, University College London, London, UK-WC1E  6BT, United Kingdom},
\mbox{ } Peter Hansbo\footnote{Department of Mechanical Engineering, J\"onk\"oping University,
SE-55111 J\"onk\"oping, Sweden}
\mbox{ } 
Mats G. Larson\footnote{Department of Mathematics and Mathematical Statistics, Ume{\aa} University, SE-90187 Ume{\aa}, Sweden}
\\
Karl Larsson\footnote{Department of Mathematics and Mathematical Statistics, Ume{\aa} University, SE-90187 Ume{\aa}, Sweden}
\mbox{ }
Andr\'e Massing\footnote{Department of Mathematics and Mathematical Statistics, Ume{\aa} University, SE-90187 Ume{\aa}, Sweden}
}

\begin{document}

\maketitle

\begin{abstract} 
\noindent
We develop a finite element method for the Laplace-Beltrami operator
on a surface with boundary and nonhomogeneous Dirichlet boundary
conditions. The method is based on a triangulation of the surface and
the boundary conditions are enforced weakly using Nitsche's method. We
prove optimal order a priori error estimates for piecewise continuous
polynomials of order $k \geq 1$ in the energy and $L^2$ norms that
take the approximation of the surface and the boundary into account.
\end{abstract}

\paragraph{Subject Classification Codes:} 65M60, 65M85.

\section{Introduction}

Finite element methods for problems on surfaces have been 
rapidly developed starting with the seminal work of Dziuk 
\cite{Dz88}. Different approaches have been developed 
including methods based on meshed surfaces, \cite{AnDeMa15}, 
\cite{DeMaSt13}, \cite{De09}, \cite{LaLa15}, and methods 
based on implicit or embedded approaches, \cite{BuHaLa15}, 
\cite{OlReGr09}, \cite{OlReXu14},
see also the 
overview articles \cite{DzEl13} and \cite{BuClHa15}, and the 
references therein. So far the theoretical developments are, 
however, restricted to surfaces without boundary. 
 
In this contribution we develop a finite element method for 
the Laplace-Beltrami operator on a surface which has a 
boundary equipped with a nonhomogeneous Dirichlet boundary 
condition. The results may be readily extended to include 
Neumann conditions on part of the boundary, which we also 
comment on in a remark. The method is based on a triangulation 
of the surface together with a Nitsche formulation \cite{Ni72} 
for the Dirichlet boundary condition. Polynomials of order 
$k$ are used both in the interpolation of the surface and 
in the finite element space. Our theoretical approach 
is related to the recent work \cite{BuHaLa15c} where a 
priori error estimates for a Nitsche method with so 
called boundary value correction \cite{BrDuTh72} is developed 
for the Dirichlet problem on a (flat) domain in $\IR^n$. We 
also mention the work 
\cite{Sc75} where the smooth curved boundary of a domain in 
$\IR^2$ is interpolated and Dirichlet boundary conditions 
are strongly enforced in the nodes.

Provided the error in the position of the approximate 
surface and its boundary is (pointwise) of order $k+1$ and 
the error in the normals/tangents is of order $k$, we prove 
optimal order error estimates in the $L^2$ and energy norms. 
No additional regularity of the exact solution, compared to 
standard estimates, is required. The proof is based on a 
Strang lemma which accounts for the error 
caused by approximation of the solution, the surface, and the 
boundary. Here the discrete surface is mapped using a closest 
point mapping onto a surface containing the exact surface. 
The error caused by the boundary approximation is then handled 
using a consistency argument. 
Special care is required to obtain optimal order $L^2$ error estimates
and a refined Aubin-Nitsche duality argument is used which exploits
the fact that the dual problem is small close to the boundary 
since the dual problem is equipped with a homogeneous Dirichlet 
condition.

The outline of the paper is as follows: In Section 2 we formulate 
the model problem and finite element method. We also formulate the 
precise assumptions on the approximation of the surface and 
its boundary. In Section 3 we develop the necessary results to 
prove our main error estimates. In Section 4 we present numerical 
results confirming our theoretical findings.

\section{Model Problem and Method}

\subsection{The Surface}

Let, $\Gamma \subset \Gamma_0$ be a surface with smooth boundary 
$\partial \Gamma$, where $\Gamma_0$ is a smooth closed connected hypersurface embedded in $\IR^3$.
We let $n$ be the exterior unit normal to $\Gamma_0$
and $\nu$ be the exterior unit conormal to $\partial \Gamma$, 
i.e. $\nu(x)$ is orthogonal both to the tangent vector of 
$\partial \Gamma$ at $x$ and the normal $n(x)$ of $\Gamma_0$.
For $\Gamma_0$, we denote its associated signed distance function by $\rho$ 
which satisfies $\nabla \rho = n$, and we define 
an open tubular neighborhood of $\Gamma_0$ by
$U_{\delta} (\Gamma_0) = \{ x \in \IR^3 : |\rho(x)|< \delta\}$
with $\delta>0$.
Then there is $\delta_{0,\Gamma_0}>0$ such that the closest point mapping $p:U_{\delta_{0,\Gamma_0}}(\Gamma_0) \rightarrow \Gamma_0$ assigns precisely 
one point on $\Gamma_0$ to each point in 
$U_{\delta_{0,\Gamma_0}}(\Gamma_0)$. The closest point mapping 
takes the form 
\begin{equation}
p:U_{\delta_{0,\Gamma_0}}(\Gamma_0) \ni x \mapsto x - \rho(x) n \circ p(x) \in \Gamma_0
\end{equation}   
For the boundary curve $\partial \Gamma$,
let $\rho_{\partial \Gamma}$ be the distance function to the 
curve $\partial \Gamma$, 
and $p_{\partial \Gamma}$ be the associated closest 
point mapping giving raise to the
tubular neighborhood $U_{\delta}(\partial \Gamma) = \{x\in \IR^3 \,:\, |\rho_{\partial \Gamma}(x)| < \delta\}$.
Note that there is $\delta_{0,\partial \Gamma}>0$ such that the
closest point mapping $p_{\partial \Gamma}:U_{\delta_{0,\partial
\Gamma}}(\partial \Gamma) \rightarrow \partial \Gamma$ is well
defined. Finally, we let $\delta_0 = \min (\delta_{0,\Gamma_0},\delta_{0,\partial \Gamma})$ and introduce
 $U_{\delta_0}(\Gamma) = \{x\in \IR^3 \,:\, |\rho(x)|\lesssim \delta_0\}$.

\begin{rem} Clearly we may take $\Gamma_0$ to be a surface that is 
only slightly larger than $\Gamma$ but for simplicity we have taken 
$\Gamma_0$ closed in order to obtain a well defined closest point 
mapping without boundary effects in a convenient way.
\end{rem}

\begin{rem} Our theoretical developments covers a smooth 
orientable hypersurface with smooth boundary in $\IR^n$, 
also for $n >3$.
\end{rem}

\subsection{The Problem}

\paragraph{Tangential Calculus.}
For each $x \in \Gamma_0$ let $T_x(\Gamma_0) 
= \{y\in \IR^3 : (y,n(x))_{\IR^3} = 0\}$ and  
$N_x(\Gamma) = \{y\in \IR^3 : \alpha n(x)), \alpha \in \IR\}$ 
be the tangent and normal spaces equipped with the inner products 
$(v,w)_{T_x(\Gamma_0)} = (v,w)_{\IR^3}$ and $(v,w)_{N_x(\Gamma_0)} 
= (v,w)_{\IR^3}$.
Let $\Ps:\IR^3 \rightarrow T_x(\Gamma_0)$ be the projection of $\IR^3$
onto the tangent space given by $\Ps = I - n \otimes n$ and let
$\Qs:\IR^3 \rightarrow N_x(\Gamma_0)$ be the orthogonal projection onto the normal space
given by $\Qs= I - \Ps = n \otimes n$.
The tangent gradient is defined by $\nablas v = \Ps \nabla v$.
For a tangential vector field $w$, i.e. a mapping $w:\Gamma_0 \ni x \mapsto
w(x) \in T_x(\Gamma_0)$,
the divergence is defined
by $\divs w = \text{tr}(w\otimes \nablas)$. 
Then the Laplace-Beltrami operator 
is given by $\Delta_\Gamma v  = \divs \nablas v$. Note that we have Green's formula 
\begin{equation}\label{eq:greens}
(-\Delta_\Gamma v,w)_{} = (\nablas v, \nablas w)_{\Gamma} 
- (\nu\cdot \nablas v, w)_{\partial \Gamma} 
\end{equation}
   
\paragraph{Model Problem.} Find $u:\Gamma \rightarrow \IR$ 
such that
\begin{alignat}{2}\label{eq:poissoninterior}
-\Delta_\Gamma u &= f \qquad 
&& \text{in $\Gamma$}
\\ \label{eq:poissonbc}
u &= g \qquad && \text{on $\partial\Gamma$}
\end{alignat}
where $f\in H^{-1}(\Gamma)$ and $g\in H^{1/2}(\partial \Gamma)$ 
are given data. Thanks to the Lax-Milgram theorem, there is a unique solution $u\in H^1(\Gamma)$ to 
this problem. Moreover, we have the elliptic regularity estimate

\begin{equation}\label{eq:ellipticregularity}
\|u\|_{H^{s+2}(\Gamma)} \lesssim \|f\|_{H^s(\Gamma)} 
+ \|g\|_{H^{s+3/2}(\Gamma)}, \qquad 
s \geq -1
\end{equation} 
since $\Gamma$ and $\partial \Gamma$ are smooth. Here and below 
we use the notation $\lesssim$ to denote less or equal up to a 
constant. We also adopt the standard notation $H^s(\omega)$ for 
the Sobolev space of order $s$ on $\omega\subset \Gamma_0$ with 
norm $\|\cdot\|_{H^s(\omega)}$. For $s=0$ we use the notation 
$L^2(\omega)$ with norm $\|\cdot\|_\omega$, see \cite{WlRoLa95} 
for a detailed description of Sobolev spaces on smooth manifolds with boundary.

\subsection{The Discrete Surface and Finite Element Spaces}

To formulate our finite element method for the boundary
value problem~\eqref{eq:poissoninterior}--\eqref{eq:poissonbc} in 
the next section, we here summarize our assumptions on the 
approximation quality of the discretization of $\Gamma$. 

\paragraph{Discrete Surface.}
Let 
$\{\Gamma_h,\, h \in (0,h_0]\}$ be a family of connected 
triangular surfaces with mesh parameter $h$ that approximates 
$\Gamma$ and let $\mcK_h$ be the mesh associated with $\Gammah$. 
For each element $K\in \mcK$, there is a bijection  
$F_K: \widehat{K} \rightarrow K$ such that $F_K \in 
[\widehat{V}_k]^3 = [P_k(\widehat{K})]^3$, where $\widehat{K}$ is a 
reference triangle in $\IR^2$ and $P_k(\widehat{K})$ is the space 
of polynomials of order less or equal to $k$. We assume that 
the mesh is quasi-uniform. 
{For each $K \in \mcK_h$, we let $n_h\vert_K$ be the unit normal to $\Gamma_h$, oriented such that 
$(n_h, n\circ p)_{\IR^3}>0$. On the element edges forming 
$\partial \Gamma_h$, we define $\nu_{\partial \Gamma_h}$ to be the exterior 
unit conormal to $\partial \Gamma_h$, i.e. $\nu_{\partial \Gamma_h}(x)$ 
is orthogonal both to the tangent vector of $\partial \Gamma_h$ at $x$ 
and the normal $n_h(x)$ of $\Gamma_h$.}
We also introduce the tangent projection $\Psh = I - n_h \otimes n_h$ and the normal projection $\Qsh = n_h \otimes n_h$, associated with $\Gammah$.

\paragraph{Geometric Approximation Property.} 
We assume that 
$\{\Gamma_h, h \in (0,h_0]\}$ approximate 
$\Gamma$ in the following way: for all $h\in (0,h_0]$ it holds
\begin{align}\label{eq:geomassum-a}
&\Gamma_h \subset U_{\delta_0}(\Gamma)
\\ \label{eq:geomassum-b}
&\partial \Gamma_h \subset U_{\delta_0}(\partial \Gamma)
\\ \label{eq:geomassum-c}
&\|\rho_\Gamma \|_{L^\infty(\Gammah)} \lesssim h^{k+1}
\\ \label{eq:geomassum-d}
&\|n \circ p_\Gamma - n_h\|_{L^\infty(\Gammah)} \lesssim h^k
\\ \label{eq:geomassum-e}
&\|\rho_{\partial \Gamma} \|_{L^\infty(\partial \Gammah)} \lesssim h^{k+1}
\\ \label{eq:geomassum-f}
&\|\nu \circ p_{\partial \Gamma} - \textcolor{black}{\nu_{\Gamma_h}}\|_{L^\infty(\Gammah)} \lesssim h^k
\end{align}
Note that it follows that we also have the estimate 
\begin{equation}\label{eq:tangentest}
\|t_{\partial \Gamma}\circ p_{\partial \Gamma} 
- t_{\partial \Gammah}\|_{L^\infty(\partial \Gammah)} 
\lesssim 
h^{k} 
\end{equation}
for the unit tangent vectors $t_{\partial \Gamma}$ and 
$t_{\partial \Gammah}$ of $\partial\Gamma$ and $\partial \Gammah$.

\paragraph{Finite Element Spaces.} Let $V_h$ be the space of parametric continuous piecewise 
polynomials of order $k$ defined on $\mcK_h$, i.e.
\begin{equation}
V_h = \{ v \in C(\Gammah,\IR)\,:\, v|_K \in \widehat{V}_k \circ F_K^{-1}\}  
\end{equation}
where $\widehat{V}_k = P_k(\widehat{K})$ is the space of polynomials of order less or equal 
to $k$ defined on the reference triangle $\widehat{K}$ defined above.

\subsection{The Finite Element Method}

The finite element method for the boundary value 
problem~\eqref{eq:poissoninterior}--\eqref{eq:poissonbc} 
takes the form: find $u_h \in V_h$ such that 
\begin{equation}\label{eq:fem}
a_{\Gammah}(u_h, v) = l_{\Gammah}(v),\qquad \forall v \in V_h
\end{equation}
where 
\begin{align}\label{eq:ahh}
a_{\Gammah}(v,w) &= (\nablash v,\nablash w)_{\Gamma_h} 
\\ \nonumber
&\qquad - (\nuhh \cdot \nablash v,w)_{\partial \Gamma_h}
-  (v, \nuhh \cdot \nablash w)_{\partial \Gamma_h}
\\ \nonumber
&\qquad + \beta h^{-1}(v,w)_{\partial \Gamma_h} 
\\ \label{eq:lhh}
l_{\Gammah}(w) &= (f\circ p ,w)_{\Gamma_h} 
- (g \circ p_{\partial \Gamma} ,\nuhh \cdot \nablash w)_{\partial \Gamma_h} 
+ \beta h^{-1}(g\circ p_{\partial \Gamma}, w)_{\partial \Gamma_h}                        
\end{align}
Here $\beta>0$ is a parameter, and $f$ is extended from $\Gamma$ 
to $\Gamma \cup p(\Gamma_h) \subset \Gamma_0$ in such a way that 
$f \in H^m(\Gamma \cup p(\Gammah))$ and
\begin{equation}\label{eq:extension-f}
\|f \|_{H^m(\Gamma \cup p(\Gammah))} \lesssim \| f \|_{H^m(\Gamma)}
\end{equation}
where $m=0$ for $k=1$ and $m=1$ for $k\geq 2$.

\begin{rem} Note that in order to prove optimal a priori error estimates 
for piecewise polynomials of order $k$ we require $u\in H^{k+1}(\Gamma)$ 
and thus $f\in H^{k-1}(\Gamma)$. For $k=1$ we have $f\in L^2(\Gamma)$ and 
for $k\geq 2$ we require $f\in H^{k-1}(\Gamma)\subseteq H^1(\Gamma)$. 
Thus we conclude that (\ref{eq:extension-f}) does not require any additional regularity compared to the standard situation. We will also 
see in Section 
\ref{sec:extension} below that there indeed exists extensions of functions that preserve regularity.
\end{rem}


\section{A Priori Error Estimates}

We derive a priori error estimates that take both the approximation 
of the geometry and the solution into account. The main new feature 
is that our analysis also takes the approximation of the boundary 
into account. 


\subsection{Lifting and Extension of Functions}

We collect some basic facts about lifting and extensions of 
functions, their derivatives, and related change of variable 
formulas, see for instance \cite{BuHaLa15}, \cite{De09}, and 
\cite{Dz88}, for further details.  

\begin{itemize}
\item For each function $v$ defined on $\Gamma_0$ we define the 
extension 
\begin{equation}
v^e = v \circ p 
\end{equation}  
to $U_{\delta_{\Gamma_0}}(\Gamma_0)$. For each function $v$ defined 
on $\Gammah$ we define the lift to $\Gammahl = p(\Gammah)\subset 
\Gamma_0$ 
by 
\begin{equation}
v^l \circ p  = v
\end{equation} 
Here and below we use the notation $\omega^l = p(\omega)\subset 
\Gamma_0$ for any subset $\omega \subset \Gammah$.

\item 
The derivative $dp:T_x(\Gammah) \rightarrow T_{p(x)}(\Gamma)$ 
of the closest point mapping $p:\Gammah \rightarrow 
\Gamma_0$ is given by
\begin{equation}\label{eq:dp}
dp(x) = P_{\Gamma}(p(x)) P_{\Gammah}(x) + \rho(x) \mathcal{H}(x)P_{\Gammah}(x)
\end{equation}
where $T_x(\Gamma)$ and $T_{p(x)}(\Gammah)$ are the tangent spaces to 
$\Gamma$ at $x \in \Gamma$ and to $\Gammah$ at $p(x)\in \Gammah$, respectively.
Furthermore,
$\mathcal{H}(x) = \nabla \otimes \nabla \rho(x)$ is the $\Gamma$ 
tangential curvature tensor which satisfies the estimate 
$\|\mathcal{H}\|_{L^\infty(U_{\delta}(\Gamma_0))} 
\lesssim 1$, for some small 
enough $\delta >0$, see \cite{GiTr01} for further details. We use $B$ 
to denote a matrix representation of the operator $dp$ with respect to 
an arbitrary choice of orthonormal bases in $T_x(\Gammah)$ and 
$T_{p(x)}(\Gamma)$.

\item Gradients of extensions and lifts are given by
\begin{equation}\label{eq:gradientidentities}
\nablash v^e = B^T \nablas v, \qquad \nablas v^l = B^{-T} \nablash v
\end{equation}
where the gradients are represented as column vectors and 
the transpose $B^T:T_{p(x)}(\Gamma)\rightarrow T_x(\Gamma_0)$ 
is defined by $(B v,w)_{T_{p(x)}(\Gamma)} 
= (v, B^T  w)_{T_x(\Gamma_0)}$, for all 
$v\in T_x(\Gammah)$ and $w\in T_{p(x)}(\Gamma)$. 

\item We have the following estimates 
\begin{equation}\label{eq:Bbounds}
\|B\|_{L^\infty(\Gamma_h)} \lesssim 1, 
\qquad 
\|B^{-1}\|_{L^\infty(\Gamma)} \lesssim 1
\end{equation}

\item We have the change of variables formulas 
\begin{equation}
\int_{\omega^l} g^l d\Gamma 
=  \int_{\omega} g |B|d\Gamma_h
\end{equation}
for a subset $\omega\subset \Gammah$, and 
\begin{equation}
\int_{\gamma^l} g^l d\Gamma 
=  \int_{\gamma} g |B_{\partial\Gammah}| 
d\Gamma_h
\end{equation}
for a subset $\gamma \subset \partial \Gamma_h$. 
Here $|B|$ denotes the absolute value of the determinant 
of $B$ (recall that we are using orthonormal bases in the 
tangent spaces) and $|B_{\partial \Gamma_h}|$ denotes 
the norm of the restriction $B_{\partial \Gammah}: T_x(\partial \Gammah) \rightarrow 
T_{p(x)}(\partial \Gammahl)$ of $B$ to the one dimensional tangent 
space of the boundary curve. We then have 
the estimates
\begin{equation}\label{eq:measurebound}
|\, |B|-1\,| \lesssim h^{k+1}, 
\qquad | \, |B^{-1}| - 1 \,| \lesssim h^{k+1}
\end{equation}
and 
\begin{equation}\label{eq:measureboundarybound}
|\, |B_{\partial \Gamma_h}| -1\,| \lesssim h^{k+1}, 
\qquad | \, 
|B^{-1}_{\partial \Gamma_h}| - 1 \,| \lesssim h^{k+1}
\end{equation}
Estimate (\ref{eq:measurebound}) appear in several papers, 
see for instance \cite{De09}. Estimate 
(\ref{eq:measureboundarybound}) is less common but 
appears in papers on discontinuous Galerkin methods on surfaces, 
see \cite{BuHaLaMa15}, \cite{DeMaSt13}, and \cite{LaLa15}. For 
completeness we include a simple proof of 
(\ref{eq:measureboundarybound}).
\paragraph{Verification of (\ref{eq:measureboundarybound}).}  
Let $\gammah:[0,a) \rightarrow \partial \Gammah \subset \IR^3$
be a parametrization of the curve $\partial \Gammah$ 
in $\IR^3$, with $a$ some positive real number. Then 
$p \circ \gammah(t)$, $t\in [0,a)$, is a parametrization of 
$\partial \Gammahl$. We have
\begin{equation}
|d_t \gammahl |_{\IR^3} = |d_t p \circ \gammah |_{\IR^3} 
= |dp d_t \gammah|_{\IR^3} 
= \muhh |d_t \gammah |_{\IR^3}
\end{equation}
and 
\begin{align}
|dp d_t \gammah|_{\IR^3} - |d_t \gammah|_{\IR^3}
&=|(\Ps + \rho \mathcal{H}) d_t \gammah |_{\IR^3} 
- |d_t \gammah|_{\IR^3}
\\
&= \underbrace{|\Ps d_t \gammah |_{\IR^3} - |d_t \gammah|_{\IR^3}}_{\bigstar = O(h^{2k})} + O(h^{k+1})
\end{align}
Here we estimated $\bigstar$ by first using the identity 
\begin{align}
| \Ps d_t \gammah|^2 &= | d_t \gammah  -\Qs d_t \gammah |^2 
\\
&= |d_t \gammah |^2 - 2 d_t \gammah \cdot \Qs d_t \gammah 
+ |\Qs  d_t \gammah|^2   
\\
&= |d_t \gammah |^2 - |\Qs  d_t \gammah|^2 
\\
&\geq ( 1- C h^{2k} ) | d_t \gammah|^2
\end{align}
and then using the estimate 
$|(1 + \delta)^{1/2} - 1|\lesssim |\delta|$, for $-1 \leq \delta$, 
to conclude that 
\begin{equation}
\Big| | \Ps d_t \gammah | - | d_t \gammah | \Big|
\lesssim h^{2k} |d_t\gammah | 
\end{equation}

%

\item The following equivalences of norms hold (uniformly in $h$)
\begin{equation}\label{eq:equivalence-H1-ext}
\| v \|_{H^m(\Gammahl)} \sim \| v^e \|_{H^m(\Gammah)},\qquad m=0,1, 
\qquad v \in H^m(\Gamma)
\end{equation}
\begin{equation}\label{eq:equivalence-H1-lift}
\| v^l \|_{H^m(\Gammahl)} \sim \| v\|_{H^m(\Gammah)},\qquad m=0,1, 
\qquad v \in H^m(\Gammah)
\end{equation}
These estimates follow from the identities for the gradients (\ref{eq:gradientidentities}), the uniform bounds (\ref{eq:Bbounds}) 
of $B$, and the bounds (\ref{eq:measurebound}) for the determinant $|B|$. 
\end{itemize}

\subsection{Norms}
We define the norms
\begin{equation}\label{eq:energynormh}
\tn v \tn^2_{\Gammah} 
= \| \nablash v \|^2_{\Gammah}
+ \tn v \tn^2_{\partial \Gammah},
\qquad
\tn v \tn^2_{\partial \Gammah}
= h \| \nuhh \cdot \nablash v \|^2_{\partial \Gammah}
+ h^{-1} \| v \|^2_{\partial \Gammah}
\end{equation}
\begin{equation}\label{eq:energynormhl}
\tn v \tn^2_{\Gammahl} 
= \| \nablas v \|^2_{\Gammahl}
+ \tn v \tn^2_{\partial \Gammahl},
\qquad \tn v \tn^2_{\partial \Gammahl}
= h \| \nuh \cdot \nablas v \|^2_{\partial \Gammahl}
+ h^{-1} \| v \|^2_{\partial \Gammahl}
\end{equation}
Here $\nuh$ denotes the unit exterior conormal to 
$\partial \Gammahl$; that is, $\nuh$ is a tangent 
vector to $\Gamma_0$, which is orthogonal to the curve $\Gammahl$ 
and exterior to $\Gammahl$. Then the following equivalences hold
\begin{equation}\label{eq:equivalence-energy-lift}
\tn v^l \tn_\Gammahl \sim \tn v \tn_\Gammah, 
\qquad \tn v^l \tn_{\partial\Gammahl} 
\sim \tn v \tn_{\partial\Gammah},
\qquad
v \in V(\Gammah)
\end{equation}
\begin{equation}\label{eq:equivalence-energy-ext}
\tn v \tn_\Gammahl \sim \tn v^e \tn_\Gammah, 
\qquad \tn v \tn_{\partial\Gammahl} 
\sim \tn v^e \tn_{\partial\Gammah},
\qquad
v \in V(\Gammahl)
\end{equation}
Here  
$V(\Gammah)=\{v\in C(\Gammah)\,:\, v|_K \in H^{3/2}(K), K \in \mcK_h\}$ 
and $V(\Gammahl) = V^l(\Gammah)$. Note that $V_h \subset V(\Gammah)$. 
\begin{rem}
We will see that it is convenient to have access to the 
norms $\tn \cdot \tn_{\partial \Gammah}$ and 
$\tn \cdot \tn_{\partial \Gammahl}$, involving the boundary terms 
since that allows us to take advantage of stronger control of 
the solution to the dual problem, which is used in the proof of 
the $L^2$ error estimate, see Theorem \ref{thm:L2}, in the vicinity 
of the boundary.  
\end{rem}

\paragraph{Verification of (\ref{eq:equivalence-energy-lift}).} In view of (\ref{eq:equivalence-H1-lift}) it is enough to verify the equivalence
$\tn v^l \tn_{\partial\Gammahl} 
\sim \tn v \tn_{\partial\Gammah}$, 
between the boundary norms.
First we have using a change of domain of integration from 
$\partial \Gammahl$ to $\partial \Gammah$ and the bound (\ref{eq:measureboundarybound}),
\begin{equation}\label{eq:equiv-ver-a}
h^{-1} \|v^l\|^2_{\partial \Gammahl} 
=   
h^{-1} ( v^l, v^l)_{\partial \Gammahl}
= 
h^{-1} ( v, v \muhh )_{\partial \Gammah}
\sim 
h^{-1} \|v\|^2_{\partial \Gammah}
\end{equation}
Next we have the identity 
\begin{equation}
\nuh \cdot \nablas v^l = \nuh \cdot B^{-T}\nablash v 
= B^{-1} \nuh \cdot \nablash v
\end{equation}
and thus using the uniform boundedness of $B^{-1}$ 
we obtain by changing domain of integration 
from $\partial\Gammahl$ to $\partial \Gammah$, 
using (\ref{eq:measureboundarybound}), and then 
splitting $\nablash v$ into components normal 
and tangent to $\partial \Gammah$, 
\begin{align}
\|\nuh \cdot \nablas v^l\|^2_{\partial \Gammahl} 
&\lesssim 
\|\nablash v\|^2_{\partial \Gammah} 
\\
&=
\|\nuhh \cdot \nablash v\|^2_{\partial \Gammah} 
+ 
\|\tth \cdot \nablash v\|^2_{\partial \Gammah}
\\
&\lesssim 
\|\nuhh \cdot \nablash v\|^2_{\partial \Gammah} 
+ 
h^{-2}\|v\|^2_{\partial \Gammah}
\\
&\lesssim 
h^{-1}\tn v \tn^2_{\partial \Gammah} 
\end{align}
where $\tth$ is the tangent vector to $\partial \Gammah$ 
and finally used an inverse estimate to bound the tangent 
derivative. Multiplying by $h$ we thus have
\begin{equation}\label{eq:equiv-ver-b}
h\|\nuh \cdot \nablas v^l\|^2_{\partial \Gammahl}
\lesssim 
\tn v \tn^2_{\partial \Gammah}
\end{equation}
The converse estimate follows by instead starting from the 
identity 
\begin{equation}
\nuhh \cdot \nablash v
=
\nuhh \cdot B^{T} B^{-T} \nablash v 
= 
 B \nuhh \cdot \nablas v^l
\end{equation}
and then using similar estimates give
\begin{equation}\label{eq:equiv-ver-c}
h \|\nuhh \cdot \nablash v \|_{\partial \Gammah}^2
\lesssim \tn v^l \tn_{\partial \Gammahl}^2
\end{equation}
Together (\ref{eq:equiv-ver-a}), (\ref{eq:equiv-ver-b}), 
and (\ref{eq:equiv-ver-c}) prove the equivalence 
$\tn v^l \tn_{\partial\Gammahl} 
\sim \tn v \tn_{\partial\Gammah}$.

\subsection{Coercivity and Continuity}

Using standard techniques, see \cite{Ni72} or  
Chapter 14.2 in \cite{LaBe13}, we find that 
$a_{\Gammah}$ is coercive
\begin{equation}\label{eq:coercivity}
\tn v \tn_{\Gammah}^2 \lesssim a_{\Gammah}(v,v)\qquad \forall 
v \in V_h
\end{equation}
provided $\beta>0$ is large enough. Furthermore, it follows directly 
from the Cauchy-Schwarz inequality that $a_\Gammah$ is continuous
\begin{equation}
a_\Gammah(v,w) \lesssim  \tn v \tn_\Gammah \tn w \tn_\Gammah 
\qquad 
\forall v,w \in V(\Gammah)
\end{equation}
Existence and uniqueness of the solution $u_h$ to the finite 
element problem (\ref{eq:fem}) follows directly from the 
Lax-Milgram lemma.

\subsection{Extension and Interpolation}
\label{sec:extension}
Next, we briefly review the fundamental interpolation estimates which will be used
throughout the remaining work. 

\paragraph{Extension.}
We note that there is an extension operator 
$E:H^s(\Gamma) \rightarrow H^s(U_{\delta_0}(\Gamma)\cap\Gamma_0)$ 
such that 
\begin{equation}\label{eq:extension-stability}
\| Ev \|_{H^s(U_{\delta_0}(\Gamma)\cap \Gamma_0)} \lesssim \| v \|_{H^s(\Gamma)}, 
\qquad s\geq 0 
\end{equation} 
This result follows by mapping to a reference neighborhood in $\IR^2$ 
using a smooth local chart and then applying the extension theorem, see \cite{Fo95}, 
and finally mapping back to the surface. For brevity 
we shall use the notation $v$ for the extended function as well, 
i.e., $v=Ev$ on $U_{\delta_0}(\Gamma)\cap \Gamma_0$. We can then 
extend $v$ to $U_{\delta_0}(\Gamma)$ by using the closest point 
extension, we denote this function by $v^e$.

\paragraph{Interpolation.}
We may now define an interpolation operator 
$\pi_h: C(U_{\delta_0}(\Gamma))\ni v \mapsto \pi_{h,L} v \in  V_h$, 
where $\pi_{h,L}$ is the nodal Lagrange interpolation operator. Consequently, 
the following interpolation error estimate holds
\begin{equation}\label{eq:interpol-basic}
\| v^e - \pi_h v^e \|_{H^m(K)} \lesssim h^{s-m}\|v \|_{H^s(K^l)}, 
\qquad 0\leq m \leq s \leq k+1
\end{equation}
Using the trace inequality to estimate 
the boundary contribution in $\tn \cdot \tn_{\Gammah}$,
\begin{equation}\label{eq:traceelement}
\| w \|^2_{\partial K} \lesssim h_K^{-1} \| w \|_K^2 
+ h_K \| \nablash w \|_K^2,\qquad v \in H^1(K), K\in \mcK_h
\end{equation}
where $h_K \sim h$ is the diameter of element $K$, 
we 
obtain
\begin{equation}\label{eq:interpolation-energy}
\tn v^e - \pi_h v^e \tn_\Gammah \lesssim h^k \| v \|_{H^{k+1}(\Gamma)}
\end{equation}
Note also that since we are concerned with smooth problems where the
solution at least resides in $H^2(\Gamma)$ and the surface is two dimensional it follows that the solution is indeed continuous from 
the Sobolev embedding theorem and therefore using the Lagrange 
interpolant is justified. We will use the
short hand notation $\pi_h^l v = (\pi_h v^e)^l$ for the lift of the
interpolant and we note that we obtain corresponding interpolation
error estimates on $\Gammah$ using equivalence of norms. We refer to
\cite{De09} and \cite{Ne76} for further details on interpolation on triangulated surfaces and \cite{Ci78} for
interpolation error estimates for the standard Lagrange 
interpolation operator.

\subsection{Strang Lemma} 

In order to formulate a Strang Lemma we first define auxiliary forms 
on $\Gammahl$ corresponding to the discrete form on $\Gammah$ as 
follows
\begin{align}\label{eq:ah}
a_{\Gammahl}(v,w) &= (\nablas v,\nablas w)_{\Gamma_h^l} 
\\ \nonumber
&\qquad - (\nuh \cdot \nablas v,w)_{\partial \Gamma_h^l}
-  (v, \nuh \cdot \nablas w)_{\partial \Gamma_h^l}
\\ \nonumber
&\qquad + \beta h^{-1}(v,w)_{\partial \Gamma_h^l} 
\\ \label{eq:lh}
l_{\Gammahl}(w) &= (f,w)_{\Gammahl} 
- (g \circ \tp ,\nuh \cdot \nablash w)_{\partial \Gammahl} 
+ \beta h^{-1}(g\circ \tp, w)_{\partial \Gammahl}           
\end{align}
Here the mapping $\tp:\partial\Gammahl \rightarrow \partial \Gamma$  
is defined by the identity 
\begin{equation}\label{def:ptilde}
\tp \circ p(x) = p_{\partial \Gamma}(x), \qquad x \in \partial \Gammah
\end{equation}
Then we find that $\tp$ is a bijection since 
$p:\partial \Gammah \rightarrow \partial \Gammahl$ and 
$p_{\partial \Gamma}: \partial \Gammah \rightarrow \partial \Gamma$ 
are bijections. Note that $a_\Gammahl$, $l_\Gammahl$, and $\tp$ are 
only used in the analysis and do not have to be implemented.

\begin{lem}\label{lem:strang} With $u$ the solution of (\ref{eq:poissoninterior}-\ref{eq:poissonbc}) and $u_h$ the solution of (\ref{eq:fem}) the following estimate holds 
\begin{align}\label{eq:strang}
\tn u - u_h^l \tn_{\Gamma_h^l} 
&\lesssim 
\tn u - (\pi_h u)^l  \tn_{\Gamma_h^l} 
\\ \nonumber 
&\qquad
+ 
\sup_{v \in V_h \setminus \{0\}} 
\frac{a_{\Gammah}( \pi_h u,v) - a_{\Gammahl}((\pi_h u)^l ,v^l)}{\tn v \tn_{\Gammah}}  
\\ \nonumber 
&\qquad
+ 
\sup_{v \in V_h \setminus \{0\}} 
\frac{l_{\Gammahl}(v^l) - l_{\Gammah}(v)}{\tn v \tn_{\Gammah}}
\\ \nonumber
&\qquad 
+ \sup_{v \in V_h \setminus \{0\}}
\frac{a_{\Gammahl}(u,v^l) - l_{\Gammahl}(v^l)}{\tn v \tn_{\Gammah}}
\end{align}
\end{lem}

\begin{rem} In (\ref{eq:strang}) the first term on the right hand side
is an interpolation error, the second and third accounts for the
approximation of the surface $\Gamma$ by $\Gammah$ and can be considered as
quadrature errors, finally the fourth term is a consistency error term
which accounts for the approximation of the boundary of the surface.
\end{rem}

\begin{proof}
We have 
\begin{align}
\tn u - u_h^l \tn_{\Gammahl} 
&\lesssim 
\tn u - (\pi_h u^e)^l \tn_{\Gammahl} + \tn (\pi_h u^e)^l- u_h^l \tn_{\Gammahl}
\end{align}
Using equivalence of norms (\ref{eq:equivalence-energy-lift}) 
and  coercivity of the bilinear form $a_h$ we have 
\begin{align}\label{eq:infsup}
\tn (\pi_h u^e)^l- u_h^l \tn_{\Gammahl}
\sim
\tn \pi_h u^e - u_h \tn_\Gammah 
&\lesssim \sup_{v \in V_h\setminus\{0\}} 
\frac{a_\Gammah( \pi_h u^e - u_h,v)}{\tn v \tn_{\Gammah}}
\end{align}
Next we have the identity 
\begin{align}\label{eq:strang-a}
a_{\Gammah}( \pi_h u^e - u_h,v) &= a_{\Gammah}( \pi_h u^e,v) - l_\Gammah(v)
\\ \label{eq:strang-b}
&= a_{\Gammah}( \pi_h u^e,v) 
- a_{\Gammahl}(u,v^l) + l_{\Gammahl}(v^l) 
- l_\Gammah(v) 
\\ \nonumber
&\qquad + a_{\Gammahl}(u,v^l) - l_{\Gammahl}(v^l)
\\ \label{eq:strang-c}
&= \underbrace{a_\Gammah( \pi_h u^e,v) 
- a_\Gammahl((\pi_h u^e)^l ,v^l)}_{I} 
+ \underbrace{l_{\Gammahl}(v^l) - l_{\Gammah}(v)}_{II}
\\ \nonumber 
&\qquad + \underbrace{a_\Gammahl((\pi_h u^e)^l - u,v)}_{III} 
+  \underbrace{a_{\Gammahl}(u,v^l) - l_{\Gammahl}(v^l)}_{IV}
\end{align}
where in (\ref{eq:strang-a}) we used the equation (\ref{eq:fem}) to eliminate $u_h$, in (\ref{eq:strang-b}) we added and subtracted $a_{\Gammahl}(u,v^l)$ and $l_{\Gammahl}(v^l)$, in 
(\ref{eq:strang-c}) we added and subtracted 
$a_\Gammahl((\pi_h u^e)^l,v)$, and rearranged the terms. Combining 
(\ref{eq:infsup}) and (\ref{eq:strang-c}) directly yields the 
Strang estimate (\ref{eq:strang}).
\end{proof}

\subsection{Estimate of the Consistency Error}
In this section we derive an estimate for the consistency error, i.e.,
the fourth term on the right hand side in the Strang Lemma
\ref{lem:strang}.
First we derive an identity for the consistency 
error in Lemma \ref{lem:identityconsistency} and then we prove 
two technical results in Lemma \ref{lem:boundaryerror} and 
Lemma \ref{lem:poincare}, and finally we give a bound of the 
consistency error in Lemma \ref{lem:consistency}. 
In order to keep track 
of the error emanating from the boundary approximation we introduce 
the notation 
\begin{equation}\label{eq:deltah}
\delta_h = \| \widetilde{\rho}_{\partial \Gamma} \|_{L^\infty(\partial \Gammahl)}\lesssim h^{k+1}
\end{equation}
where 
\begin{equation}\label{eq:tilderho}
\widetilde{\rho}_{\partial \Gamma}(x) 
= |\tp(x) - x|_{\IR^3},\qquad x \in \Gammahl
\end{equation}
The estimate in (\ref{eq:deltah}) follows from the triangle inequality 
and the geometry approximation properties (\ref{eq:geomassum-c}) and (\ref{eq:geomassum-e}). 

\begin{lem}\label{lem:identityconsistency} Let $u$ be the solution to 
(\ref{eq:poissoninterior}-\ref{eq:poissonbc}), then the 
following identity holds
\begin{align}\label{eq:consistencyidentity}
a_\Gammahl(u,v^l) - l_\Gammahl(v^l) 
&= 
-(f+\Delta_\Gamma u ,v^l)_{\Gamma_h^l \setminus \Gamma} 
\\ \nonumber
&\qquad + (u\circ \textcolor{black}{\tp} - u, \nuh \cdot \nabla_\Gamma v^l)_{\partial \Gammahl}
- \beta h^{-1} (u\circ \textcolor{black}{\tp} - u, v^l)_{\partial \Gammahl}
\end{align}
for all $v \in V_h$. 
\end{lem}
\begin{proof}
For $v \in V_h$ we have using Green's formula
\begin{align}
(f,v^l)_{\Gammahl} 
&= 
(f + \Delta_\Gamma u,v^l)_{\Gammahl} 
- (\Delta_\Gamma u,v^l)_{\Gammahl} 
\\
&=(f + \Delta_\Gamma u ,v^l)_{\Gamma_h^l \setminus \Gamma} 
+ (\nabla_\Gamma u,\nabla_\Gamma v^l)_{\Gammahl} 
- (\nuh \cdot \nabla_\Gamma u,v^l)_{\partial \Gammahl}
\\
&=(f+\Delta_\Gamma u ,v^l)_{\Gamma_h^l \setminus \Gamma} 
+ a_{\Gammahl}(u,v^l) 
+ (u,\nuh \cdot \nabla_\Gamma v^l)_{\partial \Gammahl}
- \beta h^{-1} (u,v^l)_{\partial \Gammahl}
\end{align}
where we used the fact that $f+\Delta_\Gamma u = 0$ on $\Gamma$ 
and the definition (\ref{eq:ah}) of $a_\Gammahl$. Next using the 
boundary condition $u=g$ on $\partial \Gamma$ we conclude that 
\begin{align}
(f,v^l)_{\Gammahl} 
&=(f+\Delta_\Gamma u ,v^l)_{\Gamma_h^l \setminus \Gamma} 
+ a_{\Gammahl}(u,v^l) 
+ (u,\nuh \cdot \nabla_\Gamma v^l)_{\partial \Gammahl}
- \beta h^{-1} (u,v^l)_{\partial \Gammahl}
\\ \nonumber
&\qquad - (u\circ \tp - g \circ \tp, \nuh \cdot \nabla_\Gamma v^l)_{\partial \Gammahl}
+ \beta h^{-1} (u\circ \tp - g \circ \tp, v^l)_{\partial \Gammahl}  
\end{align}
Rearranging the terms we obtain
\begin{align}\nonumber
&(f,v^l)_{\Gamma_h^l} 
- (g \circ \tp, \nuh \cdot \nabla_\Gamma v^l)_{\partial \Gammahl} 
+ \beta h^{-1} (g \circ \tp,v^l)_{\partial \Gammahl} 
\\
&\qquad=(f+\Delta_\Gamma u ,v^l)_{\Gamma_h^l \setminus \Gamma} 
+ a_{\Gammahl}(u,v^l) 
\\ \nonumber
&\qquad \qquad - (u\circ \tp - u , \nuh \cdot \nabla_\Gamma v^l)_{\partial \Gammahl}
+ \beta h^{-1} (u\circ \tp - u,v^l )_{\partial \Gammahl}  
\end{align}
where the term on the left hand side is $l_{\Gammahl}$ and the proof is 
complete.
\end{proof}
\begin{lem}\label{lem:boundaryerror} The following estimate holds
\begin{equation}\label{eq:boundaryerror}
\| v \circ \tp - v \|_{\partial \Gamma_h^l}
\lesssim 
\delta_h \| v \|_{H^2(\Gamma)},\qquad 
v \in H^2(\Gamma)
\end{equation}
where $v|_{\partial \Gammahl} = (E v)_{\partial \Gammahl}$. 
\end{lem}
\begin{proof}
For each $x \in \Gammahl$ let $I_x$ be the line segment between 
$x$ and $\tp(x)\in \partial \Gamma$, $t_x$ the 
unit tangent vector to $I_x$, and let 
$x(s) = (1 - s/{\rho_{\partial \Gamma}(x)})x + (s/\rho_{\partial \Gamma}(x)) \tp(x)$, $s\in [0,\rho_{\partial \Gamma}]$, 
be a parametrization of $I_x$. Then we have the following estimate 
\begin{align}\label{eq:consistency-c}
|v \circ \tp(x) - v(x)| 
&\lesssim \left|\int_0^{\rho_{\partial \Gamma}(x)}  \nabla v^e(x(s))\cdot t_x ds\right|
\\ \label{eq:consistency-d}
&\lesssim \|\nabla v^e \cdot t_x \|_{I_x} |\rho_{\partial \Gamma}(x)|^{1/2} 
\\ \label{eq:consistency-e}
&\lesssim \|(\nablas v)\circ p \|_{I_x} |\rho_{\partial \Gamma}(x)|^{1/2} 
\\ \label{eq:consistency-ee}
&\lesssim \|\nablas v \|_{I_x^l} |\rho_{\partial \Gamma}(x)|^{1/2} 
\end{align}
where we used the following estimates: 
(\ref{eq:consistency-d}) the Cauchy-Schwarz inequality, (\ref{eq:consistency-e}) the chain rule to conclude that 
$\nabla v^e \cdot t_x = \nabla (v\circ p) \cdot t_x 
= ((\nablas v)\circ p) \cdot dp \cdot t_x$, 
and thus we have the estimate 
\begin{equation}
\|\nabla v^e \cdot t_x \|_{I_x}\lesssim \| (\nablas v) \circ p \|_{I_x}
\end{equation}
since $dp$ is uniformly bounded in $U_{\delta_0}(\Gamma_0)$, 
(\ref{eq:consistency-ee}) changed the domain of integration 
from $I_x$ to $I_x^l = p(I_x) \subset \Gamma_0$. 
Integrating over $\partial \Gammahl$ gives
\begin{align}
\| v \circ \pb - v\|_{\partial \Gammahl}^2
&\lesssim 
\int_{\partial \Gamma_h^l} 
\| \nablas v \|^2_{I_x^l} |\rho_{\partial \Gamma}(x)| dx  
\\  \label{eq:consistency-fa}
&\lesssim 
\|\rho_{\partial \Gamma}\|_{L^\infty(\Gammahl)}
\int_{\partial \Gamma_h^l} 
\|\nablas v \|^2_{I_x^l}  dx
\\  \label{eq:consistency-f}
&\lesssim 
\delta_h
\int_{\partial \Gamma} 
\|\nablas v \|^2_{I_{y}^l} d y  
\\  \label{eq:consistency-faa}
&\lesssim 
\delta_h
\| \nablas v \|^2_{U_{\delta_h}(\partial \Gamma) \cap \Gamma_0} 
\end{align}
where we used the following estimates: 
(\ref{eq:consistency-fa}) we used H\"older's inequality, 
(\ref{eq:consistency-f}) we used the fact that 
$\|\rho_{\partial \Gamma}\|_{L^\infty(\Gammahl)} \lesssim \delta_h$ 
and changed domain of integration from $\partial \Gammahl$ to $\partial \Gamma$, 
and (\ref{eq:consistency-faa}) we integrated over a larger 
tubular neighborhood $U_{\delta_h}(\partial \Gamma)\cap \Gamma_0=\{ x \in \Gamma_0\,:\, 
|\rho_{\partial \Gamma}(x)|\lesssim \delta_h\}$ 
of $\partial \Gamma$ of thickness $2\delta_h$. We thus conclude that we have the estimate 
\begin{equation}\label{eq:consistency-faaa}
\| v \circ \pb - v\|_{\partial \Gammahl}^2 
\lesssim \delta_h \| \nablas v \|^2_{U_{\delta_h}^l(\partial \Gamma) \cap \Gamma_0} 
\end{equation}

In order to proceed with the estimates we introduce, for each 
$t\in [-\delta,\delta]$, with $\delta>0$ small enough, the 
surface 
\begin{equation}
\Gamma_t = 
\begin{cases}
\Gamma \cup (U_t(\partial \Gamma)\cap \Gamma_0) & t\geq 0
\\
\Gamma \setminus (U_t(\partial \Gamma)\cap \Gamma_0) & t<0
\end{cases}
\end{equation}
and its boundary $\partial \Gamma_t$. Starting from 
(\ref{eq:consistency-faaa}) and using H\"older's inequality in the 
normal direction we obtain
\begin{align} \label{eq:consistency-g}
\| v \circ \pb - v \|_{\partial \Gamma_h^l}
&\lesssim 
\delta_h 
\underbrace{\sup_{t\in [-\delta,\delta]} \| \nablas v \|_{\textcolor{black}{\partial \Gamma_t}}}_{\bigstar}
\\ 
&\lesssim 
\delta_h \| v \|_{H^2(\Gamma)}
\end{align}
Here we estimated $\bigstar$ using a trace inequality 
\begin{align}
\sup_{t\in [-\delta,\delta]} C_t \| \nablas v \|_{\textcolor{black}{\partial \Gamma_t}}
&\leq
\sup_{t\in [-\delta,\delta]} \| \nablas v \|_{H^1(\Gamma_t)}
\\
&\leq
\underbrace{\left(\sup_{t\in[-\delta,\delta]} C_t\right)}_{\lesssim 1} 
\| v \|_{H^2(\Gamma_{\delta})}
\\
&\lesssim \| v \|_{H^2(\Gamma)}
\end{align}
where we used the stability (\ref{eq:extension-stability}) of the 
extension of $v$ from $\Gamma_0$ to $\Gamma_{\delta}$. To see that 
the constant $C_t$ is uniformly bounded for $t\in [-\delta,\delta]$,    
we may construct a diffeomorphism $F_t:\Gamma_0 \rightarrow \Gamma_t$ 
that also maps $\partial \Gamma_0$ onto $\Gamma_t$, which has uniformly 
bounded derivatives for $t\in [-\delta,\delta]$, see the construction in \cite{BuHaLaZa15}. For $v\in H^1(\Gamma_t)$ we then have
\begin{align}
\| w \|_{\partial \Gamma_t} 
&\lesssim 
\| w \circ F_t \|_{\partial \Gamma_0}
\lesssim 
\| w \circ F_t \|_{H^1(\Gamma_0)}
\lesssim 
\| w \|_{H^1(\Gamma_t)}
\end{align}
where we used the uniform boundedness of first order derivatives of 
$F_t$ in the first and third inequality and applied a standard trace inequality on the fixed domain $\Gamma_0=\Gamma$ in the second inequality.
\end{proof}
\begin{lem}\label{lem:poincare} The following estimates hold
\begin{align}\label{eq:poincarethinbgamma}
\| v \|^2_{\Gammahl\setminus \Gamma} 
&\lesssim 
\delta_h \| v \|^2_{\partial \Gamma}
+
\delta_h^2 \| \nablas v \|^2_{\Gammahl\setminus \Gamma}
\\ \label{eq:poincarethinbgammahl}
\| v \|^2_{\Gammahl\setminus \Gamma} 
&\lesssim 
\delta_h \| v \|^2_{\partial \Gammahl}
+
\delta_h^2 \| \nablas v \|^2_{\Gammahl \setminus \Gamma}
\end{align}
for $v \in H^1(U_{\delta_0}(\partial \Gamma) \cap \Gamma_0)$ 
and $\delta_h \in (0,\delta_0]$.
\end{lem}
\begin{proof} 
Using the same notation as in Lemma \ref{lem:boundaryerror} 
and proceeding in the same way as in 
(\ref{eq:consistency-c}-\ref{eq:consistency-ee}) we  
obtain, for each $y \in I_x$, 
\begin{align}\label{eq:poincare-a}
|v(y)|
&\lesssim |v \circ \tp(x)| + \left|\int_0^{\rho_{\partial \Gamma}(y)}  \nabla v^e(x(s))\cdot t_x ds\right|
\\ \label{eq:poincare-b}
&\lesssim  |v \circ \tp(x)|  + \| \nablas v \|_{I_x^l} |\rho_{\partial \Gamma}(x)|^{1/2} 
\\ \label{eq:poincare-c}
&\lesssim  |v \circ \tp(x)|  + \delta_h^{1/2} \| \nablas v \|_{I_x^l}  
\end{align}
Integrating along $I_x$ we obtain
\begin{align} \label{eq:poincare-d}
\int_{I_x} v^2(y) dy 
& \lesssim 
\int_{I_x} 
(|v \circ \tp(x)|^2  + \delta_h \| \nablas v \|^2_{I_x^l}) dy
\\  \label{eq:poincare-e}
&\lesssim 
 \delta_h |v \circ \tp(x)|^2  + \delta_h^2 \| \nablas v \|^2_{I_x^l}
\end{align}
Finally, let $\partial \Gammahlout = \partial \Gammahl \setminus \Gamma$, 
be the part of $\partial \Gammahl$ that resides outside of $\Gamma$, 
then we have 
$\Gammahl \setminus \Gamma = \cup_{x\in \partial\Gammahlout} I_x^l$, 
and using the estimate (\ref{eq:poincare-e}) together with suitable 
changes of variables of integration we obtain
\begin{align} \label{eq:poincare-f}
\| v \|^2_{\Gammahl \setminus \Gamma} 
&\lesssim 
\int_{\partial \Gammahlout}  
\Big( \delta_h |v \circ \tp(x)|^2 
 + \delta_h^2 \| \nablas v \|^2_{I_x^l} \Big) dx
\\   \label{eq:poincare-g}
&\lesssim 
\delta_h \int_{\partial \Gammahlout} |v \circ \tp(x)|^2 dx
 + \delta_h^2 \int_{\partial \Gammahlout} \| \nablas v \|^2_{I_x^l} dx
\\  \label{eq:poincare-h}
&\lesssim 
\delta_h \|v \|_{\partial \Gamma}^2 
 + \delta_h^2 \| \nablas v \|^2_{\Gammahl \setminus \Gamma}
\end{align}
Thus the first estimate follows. The second is proved using the 
same technique.
\end{proof}
\begin{lem}\label{lem:consistency}  Let $u$ be the solution to 
(\ref{eq:poissoninterior}-\ref{eq:poissonbc}), then the 
following estimates hold
\begin{align}\label{eq:consistencyestimate-l2}
\Big|a_{\Gammahl}(u,v^l) - l_{\Gammahl}(v^l)\Big| 
&\lesssim 
\delta_h \|u\|_{H^{k+1}(\Gamma)} 
\Big( \| \nablas v^l \|_\Gammahl + h^{-1/2} \tn v \tn_{\partial \Gammahl} \Big) 
\\
&\lesssim \label{eq:consistencyestimate-energy}
h^{-1/2} \delta_h \| u \|_{H^{k+1}(\Gamma)} \tn v \tn_{\Gammah} \qquad \forall v \in V_h
\end{align}
\end{lem}
\begin{rem}
Here (\ref{eq:quadah-l2}) will be used in the proof of the 
$L^2$ norm error estimate and (\ref{eq:quadlh-energy}) in the proof 
of the energy norm error estimate. As mentioned before we will use 
stronger control of the size of solution to the dual problem, which 
is used in the proof of the $L^2$ error estimate, close 
to the boundary to handle the additional factor of $h^{-1/2}$ 
multiplying $\tn v \tn_{\partial \Gammahl}$.
\end{rem}



\begin{proof} Starting from the identity (\ref{eq:consistencyidentity}) 
and using the triangle and Cauchy-Schwarz inequalities we obtain 
\begin{align}
\Big|a_h^l(u,v^l) - l_h^l(v^l)\Big| 
&\lesssim 
\|f+\Delta_\Gamma u\|_{\Gamma_h^l\setminus \Gamma} \|v^l\|_{\Gamma_h^l \setminus \Gamma} 
\\ \nonumber
&\qquad 
+ \|u\circ \tp - u\|_{\partial \Gamma_h^l} 
\|\nuh \cdot \nabla_\Gamma v^l\|_{\partial \Gamma_h^l}
\\ \nonumber
&\qquad 
+ h^{-1} \|u\circ \tp - u\|_{\partial \Gammahl} 
\|v^l\|_{\partial \Gamma_h^l}
\\
&\lesssim 
\underbrace{\|f+\Delta_\Gamma u\|_{\Gamma_h^l\setminus \Gamma}}_{I} 
\underbrace{\|v^l\|_{\Gamma_h^l \setminus \Gamma}}_{II} 
\\ \nonumber
&\qquad 
+ \underbrace{\|u\circ \tp - u\|_{\partial \Gamma_h^l}}_{III} 
h^{-1/2} \tn v^l \tn_{\partial 
\Gammahl} 
\\ \label{eq:consistency-aa}
&\lesssim 
\underbrace{h \delta_h^{m+1/2}}_{IV\lesssim \delta_h} 
\|u\|_{H^{m+2}(\Gamma)} 
\Big( \|\nablas v^l\|_\Gammahl 
+ h^{-1/2}\tn v^l \tn_{\partial \Gammahl}\Big) 
\\ \nonumber
&\qquad + \delta_h \|u \|_{H^2(\Gamma)}
h^{-1/2} \tn v^l\tn_{\partial \Gamma_h^l}
\end{align}
for all $v \in V_h$ \textcolor{black}{and $m=0,1$}. Here we used the following estimates.
\paragraph{Term $\bfI$.} For 
$m=0$ we have using the triangle inequality, followed by the 
stability (\ref{eq:extension-f}) and (\ref{eq:extension-stability}) 
of the extensions of $f$ and $u$,
\begin{align}
\|f+\Delta_\Gamma u\|_{\Gammahl\setminus \Gamma} 
&\lesssim 
\|f\|_{\Gammahl\setminus \Gamma}  
+ \| \Delta_\Gamma u\|_{\Gammahl\setminus \Gamma}
\lesssim 
\|f\|_{\Gamma} 
+ \|u\|_{H^2(\Gamma)}
\\
&\qquad \lesssim 
\|\Delta u\|_{\Gamma} 
+ \|u\|_{H^2(\Gamma)}
\lesssim 
 \|u\|_{H^2(\Gamma)}
\end{align}
where we finally replaced $f$ by $-\Delta u$ on $\Gamma$.

For $m=1$ we note that it follows from assumption (\ref{eq:extension-f}) 
that $f + \Delta u \in H^1(\Gammah \cup \Gamma)$ and 
$f+\Delta u = 0$ on $\Gamma$, which implies $f+\Delta u = 0$ on 
$\partial \Gamma$ since the trace is well defined. We may therefore 
apply the Poincar\'e estimate (\ref{eq:poincarethinbgamma}) to extract 
a power of $\delta_h$, as follows
\begin{align}
&\|f+\Delta_\Gamma u\|_{\Gamma_h^l\setminus \Gamma} 
\lesssim 
\delta_h \|f+\Delta_\Gamma u\|_{H^1(\Gamma \textcolor{black}{\setminus} \Gammah)}
\lesssim 
\delta_h ( \|f \|_{H^1(\Gamma \cup \Gammah)} 
+ \|\Delta u  \|_{H^1(\Gamma \cup \Gammah)} )
\\ 
&\qquad 
\lesssim 
\delta_h ( \|f \|_{H^1(\Gamma)} 
+ \| u  \|_{H^3(\Gamma)} )
\lesssim 
\delta_h ( \|\Delta u \|_{H^1(\Gamma)} 
+ \| u  \|_{H^3(\Gamma)} )
\lesssim \delta_h \|u \|_{H^{3}(\Gamma)}
\end{align} 
where again we used the triangle inequality, the stability (\ref{eq:extension-f}) and 
(\ref{eq:extension-stability}), and finally replaced $f$ by 
$-\Delta u$ on $\Gamma$. 

\paragraph{Term $\bfI \bfI$.} We used the Poincar\'e estimate (\ref{eq:poincarethinbgammahl}) as 
follows
\begin{align}
\| v^l \|^2_{\Gamma_h^l \setminus \Gamma} 
&\lesssim 
\delta_h^2 \|\nablas v^l \|^2_{\Gamma_h^l \setminus \Gamma}
+ 
\delta_h \| v^l \|^2_{\partial \Gammahl} 
\\
&\lesssim 
\delta_h^2 \|\nablas v^l \|^2_{\Gamma_h^l \setminus \Gamma}
+ 
h^2 \delta_h  h^{-2}\| v^l \|^2_{\partial \Gammahl}
\\
&\underbrace{(\delta_h^2 + h^2 \delta_h )}_{\lesssim h^2 \delta_h}
\Big(\|\nablas v^l \|^2_{\Gamma_h^l \setminus \Gamma}
+h^{-2}\| v^l \|^2_{\partial \Gammahl}\Big)
\\
&\lesssim h^2 \delta_h
\Big( \|\nablas v^l\|^2_\Gammahl 
+ h^{-1}\tn v^l \tn^2_{\partial \Gammahl}\Big)
\end{align}
\paragraph{Term $\bfI\bfI\bfI$.} We used the bound (\ref{eq:boundaryerror}) 
to estimate 
$\|u\circ \tp - u\|_{\partial \Gammahl}$. 

\paragraph{Term $\bfI\bfV$.} We note that since $\delta_h \lesssim h^2$ and 
$h \in (0,h_0]$ we have $h\delta_h^{m+1/2} \lesssim \delta_h$ 
for $m=0$ and $m=1$.

This concludes the proof of estimate 
(\ref{eq:consistencyestimate-l2}). Estimate (\ref{eq:consistencyestimate-energy}) follows by a direct 
estimate of the right hand side of (\ref{eq:consistencyestimate-l2}).
\end{proof}

\subsection{Estimates of the Quadrature Errors}
\begin{lem} The following estimates hold
\begin{equation}\label{eq:technical-a}
\|\muh B^{-1} B^{-T} - P_{\Gammah} \|_{L^\infty(\Gammah)} \lesssim h^{k+1}
\end{equation}
and 
\begin{equation}\label{eq:technical-b}
\|\muhh B^{-1} \nuh  - \nuhh \|_{L^\infty(\partial \Gammah)} \lesssim h^{k+1}
\end{equation}
\end{lem}
\begin{rem} Recall that $B(x):T_x(\Gammah) \rightarrow T_{p(x)}(\Gamma)$ 
and $B^T(x):T_{p(x)}(\Gamma)\rightarrow T_x(\Gammah)$ and therefore 
$B^{-1} B^{-T}:T_x(\Gammah) \rightarrow T_x(\Gammah)$. In (\ref{eq:technical-a}) we thus estimate the deviation of 
$\muh B^{-1} B^{-T}$ from the identity $P_\Gammah$ in $T_x(\Gammah)$.
\end{rem} 

\begin{proof} {\bf (\ref{eq:technical-a}):} We have the estimate  
\begin{align}
\| \muh B^{-1} B^{-T} - P_{\Gammah} \|_{L^\infty(\Gammah)} 
&\lesssim  
\| \muh \Ps - B P_{\Gammah} B^T\|_{L^\infty(\Gamma)} 
\\ \label{eq:BBT-a}
&\lesssim \| \Ps - \Ps \Psh \Ps\|_{L^\infty(\Gamma)} + h^{k+1} 
\end{align}
where we used the uniform boundedness of $B^{-1}$, 
the identity $\muh = 1 + O(h^{k+1})$, see (\ref{eq:measurebound}), 
and, the identity $B = \Ps + O(h^{k+1})$, see (\ref{eq:dp}).
Next we have the identity
\begin{align}
\Ps - \Ps \Psh \Ps &= \Ps (I-\Psh) \Ps 
= \Ps \Qsh \Ps = (\Ps n_h)\otimes (\Ps n_h)
\end{align}
and thus 
\begin{align}\label{eq:BBT-b}
\|\Ps - \Ps \Psh \Ps\|_{L^\infty(\Gamma)}  
&\lesssim \|\Ps n_h\|^2_{L^\infty(\Gamma)} 
\lesssim \|n_h - n\|^2_{L^\infty(\Gamma)}
\lesssim h^{2k}
\end{align}
which together with (\ref{eq:BBT-a}) concludes the proof.

\paragraph{(\ref{eq:technical-b}):} Using the uniform 
boundedness of $B^{-1}$ we obtain
\begin{equation}
\|\muhh B^{-1} \nuh  - \nuhh \|_{L^\infty(\Gammah)}
\lesssim \|\muhh \nuh  - B \nuhh \|_{L^\infty(\Gammahl)}
\end{equation}
Next let $\tth$ be the unit tangent vector to $\partial \Gammah$ and 
$\tthl$ the unit tangent vector to $\partial \Gammahl$, oriented in such 
a way that \textcolor{black}{
$\nuhh = \tth \times n_h$ and $\nuh = \tthl \times n$}. We then have 
\begin{align}
B\nuhh  &= (\Ps \Psh + \rho \mathcal{H})\nuhh 
\\
&=\Ps ( \tth \times n_h ) + O(h^{k+1})
\\
&= \Ps((\Ps+\Qs)\tth \times (\Ps+\Qs)n_h) + O(h^{k+1})
\\
&= \Ps(\Ps\tth \times \Qs n_h + \underbrace{\Qs \tth \times \Ps n_h}_{O(h^{2k})}) 
+ O(h^{k+1})
\\ \label{eq:conormalbound-a}
&= \Ps\tth \times \Qs n_h + O(h^{k+1})
\end{align}
where we used the fact that $\Ps \tth\times \Ps n_h$ is normal to 
$\Gamma_0$  and $\Qs \tth\times \Qs n_h = 0$ since the vectors 
are parallel. Using (\ref{eq:conormalbound-a}) and adding and
subtracting a suitable term we obtain
\begin{align}
\muhh \nuh  - B \nuhh & = \muhh \tthl \times n 
- \Ps\tth \times \Qs n_h 
+ O(h^{k+1})
\\
&=\underbrace{(\muhh \tthl - \Ps\tth)}_{I=O(h^{k+1})}\times n 
+ \Ps \tthl \times \underbrace{(n - \Qs n_h)}_{II=O(h^{2k})} + O(h^{k+1})
\\
&=O(h^{k+1})
\end{align}
Here we used the estimates:  (\emph{I}) We have 
$\muhh \tthl = B \tth$ and thus 
\begin{equation}
\muhh \tthl - \Ps \tth 
= (B - \Ps) \tth 
= \rho \mathcal{H} \tth = O(h^{k+1})
\end{equation}
(\emph{II}) $n - \Qs n_h = (1 - n\cdot n_h) \textcolor{black}{n} 
= 2^{-1} |n - n_h|^2 \textcolor{black}{n}  = O(h^{2k})$.
\end{proof}

\begin{lem}\label{lem:quadrature} The following estimates hold
\begin{align}\label{eq:quadah-l2}
\nonumber
&\Big|a_{\Gammahl}(v^l,w^l) - a_{\Gammah}(v,w) \Big| 
\\
&\qquad \lesssim h^{k+1} 
\Big(\| \nablash v \|_{\Gammah} 
+ h^{1/2} \tn v \tn_{\partial \Gammah} 
\Big)\Big(
\| \nablash w \|_{\Gammah} 
+ 
h^{-1/2} \tn w \tn_{\partial \Gammah}
\Big)
%
 \\ \label{eq:quadah-energy}
&\qquad \lesssim h^{k+1/2} \tn v \tn_{\Gammah} \tn w \tn_{\Gammah}
\qquad 
\forall v, w \in V_h
\end{align}
and
\begin{align}\label{eq:quadlh-l2}
\Big| l_{\Gammahl}(v^l) - l_{\Gammah}(v)\Big| 
&\lesssim h^{k+1} \Big( \|f\|_\Gamma  + \| g \|_{\partial \Gamma} \Big) 
\Big(\| \nablash v \|_\Gammah 
+ h^{-1/2}\tn v \tn_{\partial \textcolor{black}{\Gammah}} \Big)
\\ \label{eq:quadlh-energy}
&\lesssim h^{k+1/2} \Big( \|f\|_\Gamma + \|g\|_{\partial\Gamma} \Big) 
\tn v \tn_\Gammah 
\qquad \forall v\in V_h 
\end{align}
\end{lem}
\begin{rem} In fact the estimate (\ref{eq:quadah-energy}) holds also 
with the factor $h^{k+1}$, which is easily seen in the proof below. However, (\ref{eq:quadah-energy}) is only used in the proof of the 
energy norm error estimate which is of order 
$h^k$ so there is no loss of order. We have chosen this form since it 
is analogous with the estimates of the right hand side 
(\ref{eq:quadlh-l2})-\textcolor{black}{(\ref{eq:quadlh-energy})}. 
\end{rem}
\begin{rem} We note that the estimates in Lemma \ref{lem:quadrature} 
have similar form as the estimates in Lemma \ref{lem:consistency}, 
which are adjusted to fit the $L^2$ and energy norm estimates. 
\end{rem}

\begin{proof}{\bf (\ref{eq:quadah-l2})-(\ref{eq:quadah-energy}):} 
Starting from the definitions of the forms (\ref{eq:ahh}) and 
(\ref{eq:ah}) we obtain
\begin{align}
a_{\Gammahl}(v^l,w^l) - a_{\Gammah}(v,w)
&= (\nablas v^l,\nablas w^l)_{\Gamma_h^l} 
              - (\nablash v,\nablash w)_{\Gamma_h}
\\ \nonumber
&\qquad - (\nuh \cdot \nablas v^l,w^l)_{\partial \Gamma_h^l} 
+(\nuhh \cdot \textcolor{black}{\nablash} v,w)_{\partial \Gamma_h}
\\ \nonumber
&\qquad 
 -  (v^l, \nuh \cdot \nablas w^l)_{\partial \Gamma_h^l}
+  (v, \nuhh \cdot \textcolor{black}{\nablash} w)_{\partial \Gamma_h}
\\ \nonumber
&\qquad + \beta h^{-1}\Big( (v^l,w^l)_{\partial \textcolor{black}{\Gamma_h^l}} 
- (v,w)_{\partial \Gamma_h} \Big)
\\
&= I + II + III + III
\end{align}
\paragraph{Term $\bfI$.} We have the estimates
\begin{align}
|I| & = \Big|( B^{-T} \nablash v, B^{-T} \nablash w \textcolor{black}{\muh} )_{\Gammah} 
- (\nablash v, \nablash w)_{\Gammah}\Big|
\\
&=\Big|( (\muh B^{-1} B^{-T} - P_{\Gammah}) \nablash v, \nablash w  )_{\Gammah}\Big|
\\
&\lesssim 
h^{k+1} \|\nablash v \|_{\Gammah} \|\nablash w \|_{\Gammah}
\end{align}
where we used the estimate (\ref{eq:technical-a}). 

\paragraph{Terms $\bfI\bfI$ and $\bfI\bfI\bfI$.}
Terms $II$ and $III$ have the same form and may be estimated as follows 
\begin{align}
|II| &= \Big|(\nuh \cdot \nablas v^l,w^l)_{\partial \Gamma_h^l} 
\textcolor{black}{-}(\nuhh \cdot \textcolor{black}{\nablash} v,w)_{\partial \Gamma_h}\Big|
\\
&=\Big|(\nuh \cdot B^{-T} \nablash v,w \muhh )_{\textcolor{black}{\partial \Gamma_h}} \textcolor{black}{-}
  (\nuhh \cdot \textcolor{black}{\nablash},w)_{\partial \Gamma_h}\Big|
\\
&=\Big|((\muhh B^{-1} \nuh - \nuhh) \textcolor{black}{\cdot} \nablash v,w )_{\textcolor{black}{\partial \Gamma_h}} \Big|
\\
&\leq \| \muhh B^{-1} \nuh - \nuhh \|_{L^\infty(\partial \Gammah)} 
\| \nablash v \|_{\partial \Gammah} \| w \|_{\partial \Gammah}
\\
&\lesssim h^{k+1} h^{1/2}\tn v \tn_{\textcolor{black}{\Gammah}} 
h^{-1/2}\tn w \tn_{\partial \Gammah}  
\end{align}
where we used (\ref{eq:technical-b}) and the inverse estimate 
\begin{equation}
h \| \nablash v \|^2_{\textcolor{black}{\partial}\Gammah} 
\lesssim \|\nablash v \|^2_{\mcK_h(\Gammah)}
\lesssim \| \nablash v \|^2_{\Gammah}
\end{equation}
for all $v \in V_h$. Thus we conclude that 
\begin{equation}
|II| + |III| 
\lesssim h^{k+1} h^{1/2}\tn v \tn_{\textcolor{black}{\Gammah}} 
h^{-1/2}\tn w \tn_{\partial \Gammah}
\end{equation}

\paragraph{Term $\bfI\bfV$.} We have 
\begin{align}
|IV| &= \beta h^{-1}\Big| (v^l,w^l)_{\partial \Gammahl} 
- (v,w)_{\partial\Gammah}\Big|
\\
&=  \beta h^{-1}\Big|((\muhh - 1) v,w)_{\partial \Gammah}\Big|
\\
& \lesssim h^{-1}\|\muhh - 1\|_{L^\infty(\partial \Gammah\textcolor{black}{)}} 
\|v\|_{\partial \Gammah} \|w\|_{\partial \Gammah}
\\
&\lesssim h^{k+1} h^{1/2}\tn v \tn_{\partial\Gammah} 
h^{-1/2}\tn w \tn_{\partial \Gammah}
\end{align}
\textcolor{black}{Estimate} (\ref{eq:quadah-energy}) \textcolor{black}{follows} by a direct estimate 
of the right hand side of (\ref{eq:quadah-l2}).

\paragraph{(\ref{eq:quadlh-l2}) and (\ref{eq:quadlh-energy}):} We have
\begin{align}
\Big|l_\Gammahl(w^l) - l_\Gammah(w)\Big|
&= \Big|(f,w^l)_{\Gammahl} - (f\circ p_\Gamma,w)_\Gammah 
\\ \nonumber
&\qquad - (g \circ \tp ,\nuh \cdot \nablas w^l)_{\partial \Gammahl} 
+ (g \circ p_{\partial \Gamma} ,\nuhh \cdot \nablash w)_{\partial \Gammah} 
\\ \nonumber
&\qquad + \beta h^{-1}(g\circ \tp, w^l)_{\partial \Gammahl}                        
 -  \beta h^{-1}(g\circ p_{\partial \Gamma}, w)_{\partial \Gammah}                    
 \Big|
 \\ 
 &\leq \Big|(\textcolor{black}{\muh} -1) f\circ p_\Gamma,w)_\Gammah\Big| 
\\ \nonumber
&\qquad + \Big| (g \circ p_{\partial \Gamma} , 
(\textcolor{black}{\muhh B^{-1}} \nuh  - \nuhh) 
\cdot \nablash w)_{\partial \Gammah}\Big| 
\\ \nonumber
&\qquad + \beta h^{-1}\Big|((\textcolor{black}{\muhh}-1)g\circ p_{\partial \Gamma}, w)_{\partial \Gammah}                    
 \Big|
 \\
  &\lesssim 
  h^{k+1} \|f \|_\Gamma \|w\|_\Gammah 
  + h^{k+1} \|g \|_{\partial \Gamma} \|\nablash w\|_{\textcolor{black}{\partial}\Gammah}
  + h^{k}\|g \|_{\partial \Gamma} \|w\|_{\partial \Gammah}                  
\end{align}
\textcolor{black}{where we used (\ref{eq:measurebound}), (\ref{eq:technical-b}) and (\ref{eq:measureboundarybound}).} Next using the Poincar\'e estimate 
\begin{equation}
\| w \|_\Gammah 
\lesssim \|\nablash w \|_\Gammah + \| w \|_{\partial \Gammah}
\lesssim \|\nablash w \|_\Gammah + h^{1/2}\tn w \tn_{\partial \Gammah}
\end{equation}
we obtain
\begin{align}
\Big|l_\Gammahl(w^l) - l_\Gammah(w)\Big|
&\lesssim   h^{k+1} \|f \|_\Gamma \|w\|_\Gammah 
\\ \nonumber
&\qquad 
  + h^{k+1} \|g \|_{\partial \Gamma} h^{-1/2}\tn  w \tn_{\textcolor{black}{
  \partial \Gammah}} 
   + \textcolor{black}{h^{k} \|g \|_{\partial \Gamma} h^{1/2}\tn  w \tn_{\partial
  \Gammah}}
\\
&\lesssim   h^{k+1} \|f \|_\Gamma 
\Big( \|\nablash w \|_\Gammah + h^{1/2}\tn w \tn_{\partial \Gammah} \Big)
\\ \nonumber
&\qquad  + h^{k+1} \|g \|_{\partial \Gamma} h^{-1/2}\tn w \tn_{\textcolor{black}{ \partial \Gammah}}
\\
&\lesssim 
h^{k+1}\Big( \|f \|_\Gamma + \|g \|_{\partial \Gamma}\Big)
\Big( \|\nablash w \|_\Gammah + h^{-1/2}\tn w \tn_{\partial \Gammah} \Big)
\\
&\lesssim 
h^{k+1/2}\Big( \|f \|_\Gamma + \|g \|_{\partial \Gamma}\Big)
\tn w \tn_\Gammah 
\end{align}
which are the desired estimates.
\end{proof}

\subsection{Error Estimates}
With the abstract Strang Lemma~\ref{lem:strang} and the estimates for
the interpolation, quadrature and consistency error, we are now
prepared to prove the main a priori error estimates.
\begin{thm} \label{thm:energy}
With $u$ the solution of (\ref{eq:poissoninterior}-\ref{eq:poissonbc}) and $u_h$ the solution of (\ref{eq:fem}) the 
following estimate holds 
\begin{equation}\label{eq:energynormerrorest}
\tn u - u_h^l \tn_{\Gammahl} 
\lesssim h^k  \Big(\| u \|_{H^{k+1}(\Gamma)} 
+ \|f\|_{\Gamma} + \|g\|_{\partial \Gamma} \Big) 
\end{equation}
\end{thm}
\begin{proof} Starting from the Strang Lemma and using the 
interpolation estimate (\ref{eq:interpolation-energy}), the 
quadrature error estimates (\ref{eq:quadah-energy}) and 
(\ref{eq:quadlh-energy}), and the consistency error estimate 
(\ref{eq:consistencyestimate-energy}), we obtain
\begin{align}
\tn u - u_h^l \tn_{\Gammahl} 
& \lesssim h^k \| u \|_{H^{k+1}(\Gamma)}
+ h^{k+1/2} \tn \pi_h u^e \tn_{\Gammah}
+ h^{k+1/2} \Big( \|f\|_{\Gamma} + \| g \|_{\partial \Gamma} \Big)
\\ \nonumber
&\qquad + h^{-1/2}\delta_h \| u \|_{H^2(\Gamma)}
\\ \label{eq:thmenergy-a}
& \lesssim h^k \| u \|_{H^{k+1}(\Gamma)}
+ h^{k+1/2} \Big( \|f\|_{\Gamma} + \| g \|_{\partial \Gamma} \Big)
\\ \nonumber
&\qquad + h^{k+1/2} \| u \|_{H^2(\Gamma)}
\end{align}
Here, in (\ref{eq:thmenergy-a}), we used the estimate
\begin{align}
\tn \pi_h u^e \tn_{\Gammah} 
&\lesssim 
\tn \pi_h u^e - u^e \tn_{\Gammah}
+ 
\tn u^e \tn_{\Gammah}
\\ \label{eq:thmenergy-b}
&\lesssim 
h^k\|u\|_{H^{k+1}(\Gamma)} 
+ 
h^{-1/2} \| u \|_{H^2(\Gamma)} 
\end{align}
where, in (\ref{eq:thmenergy-b}), we used the interpolation 
estimate (\ref{eq:interpolation-energy}) to estimate the 
first term and a trace inequality to estimate the second 
term, and finally the inequality
$h^{-1/2} \delta_h \lesssim h^{k+1/2}$. Thus the proof is 
complete since $k\geq 1$ and $h\in(0,h_0]$.
\end{proof}

\begin{thm} \label{thm:L2}
With $u$ the solution of (\ref{eq:poissoninterior}-\ref{eq:poissonbc}) and $u_h$ the solution of (\ref{eq:fem}) the 
following estimate holds 
\begin{equation}
\| u - u_h^l \|_{\Gammahl}  
\lesssim h^{k+1}  \Big(\| u \|_{H^{k+1}(\Gamma)} 
+ \|f\|_{\Gamma} + \|g\|_{\partial \Gamma}\Big)
\end{equation}
\end{thm}
\begin{proof} Let $\phi \in H^1_0(\Gamma)$ be the solution 
to the dual problem
\begin{equation}\label{eq:dual}
a(v,\phi) = (v,\psi), \qquad v \in H^1_0(\Gamma)
\end{equation}
where $\psi=e = u -u_h^l$ on $\Gamma_h^l$ and $\psi=0$ on 
$\Gamma \setminus \Gamma_h^l$, and extend $\phi$ using 
the extension operator to $U_{\delta_0}(\Gamma)\cap \Gamma_0$. 
Then we have the stability estimate
\begin{equation}\label{eq:dualstab}
\|\phi \|_{H^2(\Gamma \cup \Gammahl\textcolor{black}{)}}
\lesssim 
\| \phi \|_{H^2(\Gamma)} 
\lesssim 
\| \psi \|_{\Gammahl}
= \|e \|_{\Gammahl}
\end{equation}
where the first inequality follows from the stability (\ref{eq:extension-stability}) of the extension of $\phi$ and the second is the elliptic regularity of the solution to the dual problem.

We obtain the following representation formula 
for the error  
\begin{align}
\|e\|^2_{\Gammahl} &= (e,\psi+\Delta \phi)_{\Gammahl} 
- (e,\Delta \phi)_{\Gammahl}
\\
&=(e,\psi +\Delta \phi)_{\Gammahl \setminus \Gamma}
+ (\nabla e,\nabla \phi )_{\Gammahl} 
- (e,\nuh \cdot \nabla \phi)_{\partial \Gammahl}
\\
&=\underbrace{(e,\psi +\Delta \phi)_{\Gammahl \setminus \Gamma}}_{I}
+ \underbrace{a_{\Gammahl}(e,\phi)}_{II} 
+ \underbrace{(\nuh \cdot \nablas e, \phi)_{\partial \Gammahl}
- \beta h^{-1} (e,\phi)_{\partial \Gammahl}}_{III}
\end{align}

\paragraph{Term $\bfI$.} We have the estimates
\begin{align}
|I|&=|(e,\psi+\Delta \phi)_{\Gammahl \setminus \Gamma}|
\\
&\lesssim \|e\|_{\Gammahl \setminus \Gamma} 
\|\psi+\Delta \phi\|_{\Gammahl \setminus \Gamma}
\\
&\lesssim 
\Big(\delta_h^2 \|\nablas e \|^2_{\Gammahl \setminus \Gamma}
+ \delta_h \|e\|^2_{\partial \Gammahl} \Big)^{1/2}
\Big(\| \psi \|_{\Gammahl \setminus \Gamma} +  
\| \Delta \phi \|_{\Gammahl \setminus \Gamma} \Big)
\\
&\lesssim 
\Big( (\delta_h^2 + h \delta_h) \tn e \tn_\Gammahl^2 \Big)^{1/2}
\Big( \|e\|_{\Gammahl \setminus \Gamma} + \| \phi \|_{H^2(\Gamma)} \Big)
\\ \label{est:L2I}
&\lesssim 
\underbrace{(h^{-2} \delta_h + h^{-1}\delta_h)^{1/2}}_{\lesssim 1} h 
\tn e \tn_\Gammahl \| e \|_{\Gammahl}
\end{align}
Here we used the Poincar\'e estimate (\ref{eq:poincarethinbgammahl}) 
together with the definition of the energy norm to conclude that 
$\|e\|_{\Gammahl\setminus \Gamma} \lesssim h \tn e \tn_\Gammahl$, 
the stability (\ref{eq:dualstab}) of the dual problem to conclude 
that $\|\psi + \Delta \phi \|_{\Gammahl\setminus \Gamma} 
\lesssim \| e \|_{\Gammahl}$, 
and finally the fact $\delta_h \lesssim h^{k+1}$.

\paragraph{Term $\bfI\bfI$.} Adding and subtracting an interpolant 
we obtain
\begin{align}
|II| &= |a_\Gammahl(e,\phi - \pi_h^l \phi) + a_\Gammahl(e,\pi_h^l \phi)|
\\
&\lesssim 
\tn e \tn_\Gammahl \tn \phi - \pi_h^l \phi \tn_\Gammahl 
+
|a_\Gammahl(e,\pi_h^l \phi)|
\\
&\lesssim 
h \tn e \tn_\Gammahl \|\phi\|_{H^2(\Gamma)}
+ |a_\Gammahl(e,\pi_h^l \phi)|
\\ \label{proof:L2a}
&\lesssim 
h \tn e \tn_\Gammahl \| e\|_{\Gammahl}
+ |a_\Gammahl(e,\pi_h\phi)|
\end{align}
For the second term on the right hand side we first note that using
\textcolor{black}{Lemma \ref{lem:consistency} and Lemma \ref{lem:quadrature}}
we have the estimates
\begin{align}
a_\Gammahl(e,\pi_h \phi) & = a_\Gammahl(u,\pi_h^l \phi) - 
a_\Gammahl(u_h^l,\pi_h^l \phi)
\\
& = a_\Gammahl(u,\pi_h^l \phi) - l_\Gammahl(\pi_h^l \phi)
\\ \nonumber  
&\qquad + l_\Gammahl(\pi_h^l \phi) - l_\Gammah(\pi_h \phi)
\\ \nonumber
&\qquad + a_\Gammah(u_h,\pi_h \phi) - a_\Gammahl(u_h^l,\pi_h \phi)
\\
& = \delta_h \| u \|_{H^{k+1}(\Gamma)}
\Big(\| \nablas \pi_h^l \phi \|_{\Gammahl} 
+ h^{-1/2} \tn \pi_h^l \phi \tn_{\partial \Gammahl} \Big)
\\ \nonumber  
&\qquad + h^{k+1} \Big( \| f \|_{\Gamma} + \|g \|_{\partial \Gamma} \Big)
\Big(\| \textcolor{black}{\nablash} \pi_h \phi \|_{\textcolor{black}{\Gammah}} 
+ h^{-1/2} \tn \pi_h \phi \tn_{\textcolor{black}{\partial \Gammah}} \Big)
\\ \nonumber
&\qquad + h^{k+1} \Big(\| \nablash u_h \|_\Gammah 
+ h^{1/2}\tn u_h \tn_{\partial \Gammah} \Big)
\Big(
\| \nablash \pi_h \phi \|_\Gammah
+
h^{-1/2}  \tn \pi_h \phi \tn_{\partial \Gammah} \Big)
\\
&\lesssim \delta_h \| u \|_{H^{k+1}(\Gamma)} \|e\|_{\Gammahl}
+ h^{k+1} \Big( \| f \|_{\Gamma} + \|g \|_{\partial \Gamma} \Big)
\|e\|_{\Gammahl}
\end{align}
Here we used the estimate
\begin{align}\label{uhstab1}
\| \nablash u_h \|_\Gammah 
+ h^{1/2}\tn u_h \tn_{\partial \Gammah} 
&\sim
\| \nablas u_h^l \|_\Gammahl 
+ h^{1/2}\tn u_h^l \tn_{\partial \Gammahl}
\\
&\leq
\| \nablas (u_h^l - u) \|_\Gammahl 
+ h^{1/2}\tn (u_h^l - u) \tn_{\partial \Gammahl}
\\ \nonumber
&\qquad + 
\| \nablas u \|_\Gammahl 
+ h^{1/2}\tn u \tn_{\partial \Gammahl}
\\
&\lesssim h^k \| u \|_{H^{k+1}(\Gamma)} + \| u \|_{H^2(\Gamma)} \label{uhstab3}
\end{align}
where we added and subtracted the exact solution, used the triangle 
inequality and the energy norm error estimate (\ref{eq:energynormerrorest}) 
and finally a trace inequality to estimate the last term.
For the dual problem we obtain
\[ \textcolor{black}{
\|\nablash \pi_h \phi \|_{\textcolor{black}{\Gammah}} 
+ h^{-1/2} \tn \pi_h \phi \tn_{\partial \Gammah} \sim \| \nablas \pi_h^l \phi  \|_\Gammahl 
+ h^{-1/2}\tn \pi_h^l \phi \tn_{\partial \Gammahl}}.
\]
\textcolor{black}{
The first term of the right hand side is handled as in
(\ref{uhstab1})-(\ref{uhstab3}) and the second is bounded as follows}
\begin{align}
h^{-1}\tn \pi_h^l \phi \tn^2_{\partial \Gammahl}
&\lesssim 
h^{-1}\tn \pi_h^l \phi - \phi  \tn^2_{\partial \Gammahl}
+ 
h^{-1}\tn \phi \tn^2_{\partial \Gammahl}
\\
&\lesssim 
h \|\phi \|\textcolor{black}{^2}_{H^2(\Gammahl)}
+ 
\| \nuh \cdot \nablas \phi \|^2_{\partial \Gammahl}
+ h^{-2} \| \phi \|^2_{\partial \Gammahl}
\\
&\lesssim 
h \|\phi \|\textcolor{black}{^2}_{H^2(\Gammahl)}
+ 
\| \phi \|^2_{H^2(\Gammahl)}
+ h^{-2}\delta_h^2 \| \phi \|^2_{H^2(\Gamma)}
\\
&\lesssim \underbrace{(h + 1 + h^{-2}\delta_h^2)}_{\lesssim 1} 
\|\phi \|\textcolor{black}{^2}_{H^2(\Gamma)}
\end{align}
where we added and subtracted the exact solution, used the 
interpolation error estimate (\ref{eq:interpolation-energy}) for the 
first term on the right hand side, a trace inequality for the second 
term, the fact that $\phi=0$ on $\Gamma$ together with (\ref{eq:boundaryerror}) for the third term, and finally stability 
of the extension operator (\ref{eq:extension-stability}).
Thus we conclude that 
\begin{equation}
\|\nablas \pi_h^l \phi \|_\Gammahl + h^{-1/2}\tn \pi_h^l \phi \tn_{\partial \Gammahl}
\lesssim \|\phi\|_{H^2(\Gamma)} 
\lesssim \|e\|_{\Gammahl}
\end{equation}
%
%

%
\paragraph{Term $\bfI\bfI\bfI$.} Using the Cauchy-Schwarz inequality 
we get
\begin{align}
|III| &\leq h\tn e \tn_\Gammahl h^{-3/2}\|\phi\|_{\partial \Gammahl} 
\lesssim h\tn e \tn_\Gammahl h^{-3/2}\delta_h \|\phi\|_{H^2(\Gamma)}
\lesssim h\tn e \tn_\Gammahl \| e \|_{\Gammah}
\end{align}

\begin{rem} Our results directly extends to the case of a Neumann or 
Robin condition 
\begin{equation}
\nu \cdot \nablas u = g_N - \kappa u
\end{equation}
where $\kappa\geq 0$ on a part of the boundary. Essentially we need to modify the quadrature term estimates to account 
for the terms involved in the weak statement of the Robin. 
These terms are very similar to the terms involved in the Nitsche formulation for the Dirichlet problem and may be estimated in the same way.
\end{rem}


\section{Numerical Examples}

We consider the Laplace-Beltrami problem on a torus with 
a part removed. To express points on the torus surface we use
toroidal coordinates $\{\theta,\phi\}$ defined such that the 
corresponding Cartesian coordinates are given by
\begin{align}
x = (R+r \cos(\theta)) \cos(\phi), \quad
y = (R+r \cos(\theta)) \sin(\phi), \quad 
z = r \sin(\theta)
\end{align}
with constants $R=1$ and $r=0.4$. The boundary 
$\partial\Gamma$ is defined by the curves
\begin{align}\label{eq:numerics-boundary-a}
\phi_1(\theta) = 0.2 \cos (N_1 \theta)
\qquad
\text{and}
\qquad
\phi_2(\theta) = 0.2 \cos (N_2 \theta) + 0.6 (2R\pi)
\end{align}
where we choose $N_1=4$ and $N_2=3$. In turn the domain $\Gamma$ 
is given by
\begin{align}
\Gamma = \left\{ \theta,\phi \ : \ 0 \leq \theta < 2\pi \, , \ \phi_1 \leq \phi \leq \phi_2 \right\}
\end{align}
We manufacture a problem with a known analytic solution by 
prescribing the solution
\begin{align}
u = \cos(3\phi+5\theta)\sin(2\theta)
\label{eq:uansatz}
\end{align}
and compute the corresponding load $f$ by using the identity
$f=-\Delta_\Gamma u$. The non-homogenous Dirichlet boundary conditions
on $\partial\Gamma$ are directly given by $u|_{\partial\Gamma}$.  Note
that \eqref{eq:uansatz} is smooth and defined on the complete torus so
clearly the stability estimates \eqref{eq:extension-f} and
\eqref{eq:extension-stability} for $f$ and $u$ both hold.

\paragraph{Geometry Discretization $\boldsymbol{\Gamma}_{\boldsymbol h}$.}
We construct higher order ($k>1$) geometry approximations $\Gammah$
from an initial piecewise linear mesh ($k=1$) by adding nodes for
higher order Lagrange interpolation through linear interpolation
between the facet vertices.  All mesh nodes are moved to the exact
surface by the closest point map $p(x)$ and then the boundary is
corrected such that the nodes on the discrete boundary
$\partial\Gamma_h$ coincide with the exact boundary
$\partial\Gamma$. A naive approach for the correction is to just move
nodes on the boundary of the mesh to the exact boundary. For our model
problem this is done through the map
$\widetilde{p}_{\partial\Gamma}:\partial\Gammahl
\rightarrow \partial\Gamma$ given by
$\widetilde{p}_{\partial\Gamma}=\phi_i$.  This may however give
isoparametric mappings with bad quality or negative Jacobians in some
elements, especially in coarser meshes and higher order interpolations
where the element must be significantly deformed to match the
boundary.  We therefore use a slightly more refined procedure where
interior nodes are placed inside the element according to a quadratic
map aligned to the boundary, rather than using linear interpolation
over the facet.  In Figure~\ref{fig:mesh} a coarse mesh for the model
problem using $k=3$ interpolation is presented.

\paragraph{Numerical Study.} The numerical solution for the model problem with $k=3$ and $h=1/4$ is visualized in Figure~\ref{fig:solution}. We choose the Nitsche penalty parameter $\beta=10^4$. This large value was chosen in order to achieve the same size of the error as when strongly enforcing the Dirichlet boundary conditions and using $k=4$.

The results for the convergence studies in energy norm and $L^2$ norm are presented in Figure~\ref{fig:convenergy} and Figure~\ref{fig:convL2}. These indicate convergence rates of $O(h^k)$ in energy norm and $O(h^{k+1})$ in $L^2$ norm which by norm equivalence is in agreement with
Theorem~\ref{thm:energy} and Theorem~\ref{thm:L2}, respectively.
For coarse meshes we note some inconsistency of the results in energy norm for higher order interpolations. We attribute this effect to large derivatives of the mappings used to make the element fit the boundary which may arise in some elements for coarse meshes that are large in 
comparison to the variation of the boundary. When the boundary is better resolved we retain the proper convergence rates. Note also that the Jacobian of the mapping is involved in the computation of the gradient 
which explains that we see this behavior in the energy norm but not in 
the $L^2$ norm.

In the special case $\Gammahl = \Gamma$, such as the simplified model problem, obtained by taking parameters $N_1=N_2=0$ in the boundary description (\ref{eq:numerics-boundary-a}), illustrated by the mesh in Figure~\ref{fig:mesh_simplified}, no correction of boundary nodes onto 
$\partial\Gamma$ is needed. In that case the energy error aligns perfectly with the reference lines also for coarse meshes and higher order geometry approximations, see Figure~\ref{fig:convenergy_simplified}.

\begin{figure}
\centering
\includegraphics[width=0.60\linewidth]{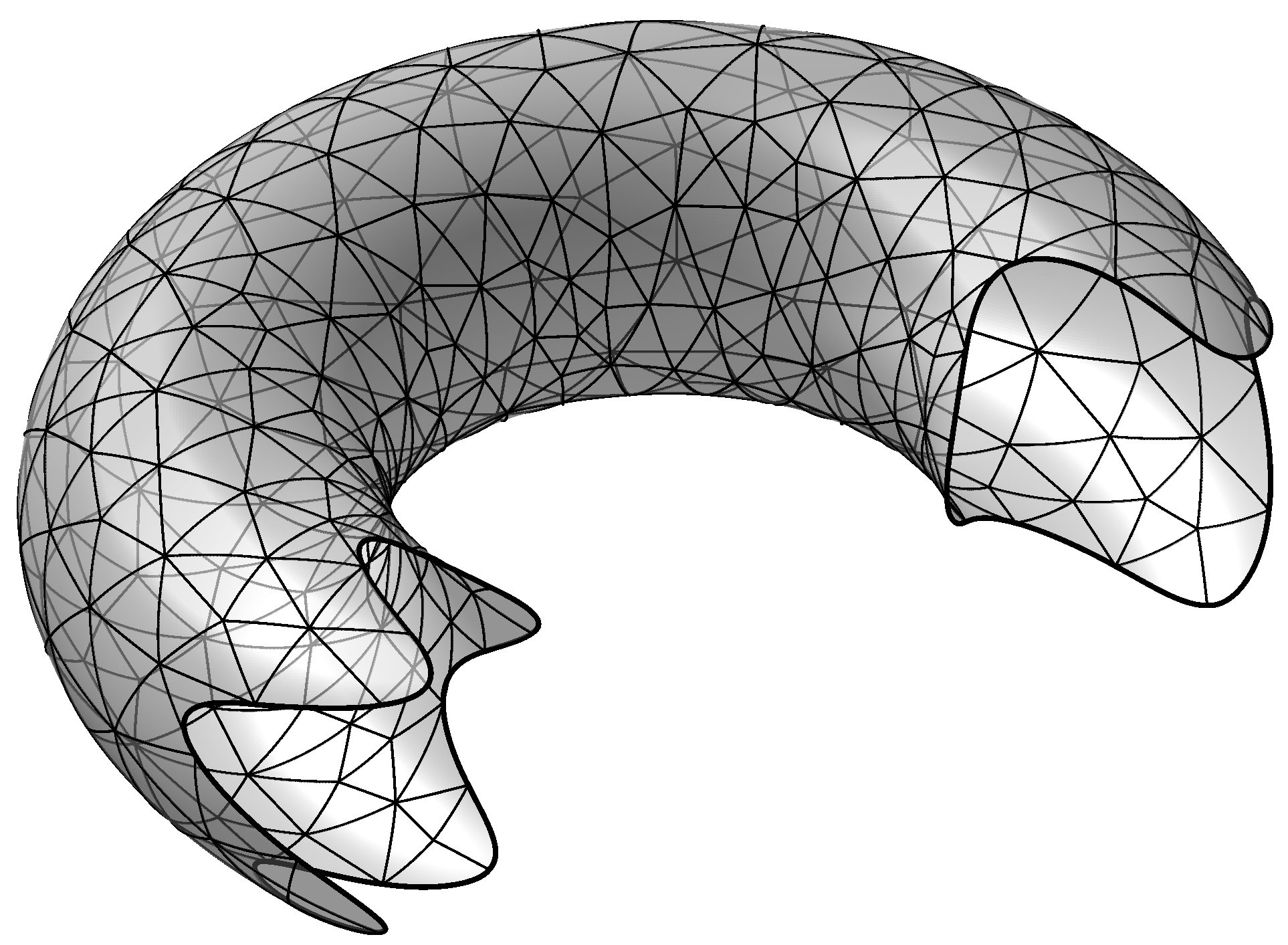}
\caption{Mesh for the model problem using geometric interpolation order $k=3$ and meshsize $h=1/4$.
}
\label{fig:mesh}
\end{figure}

\begin{figure}
\centering
\includegraphics[width=0.60\linewidth]{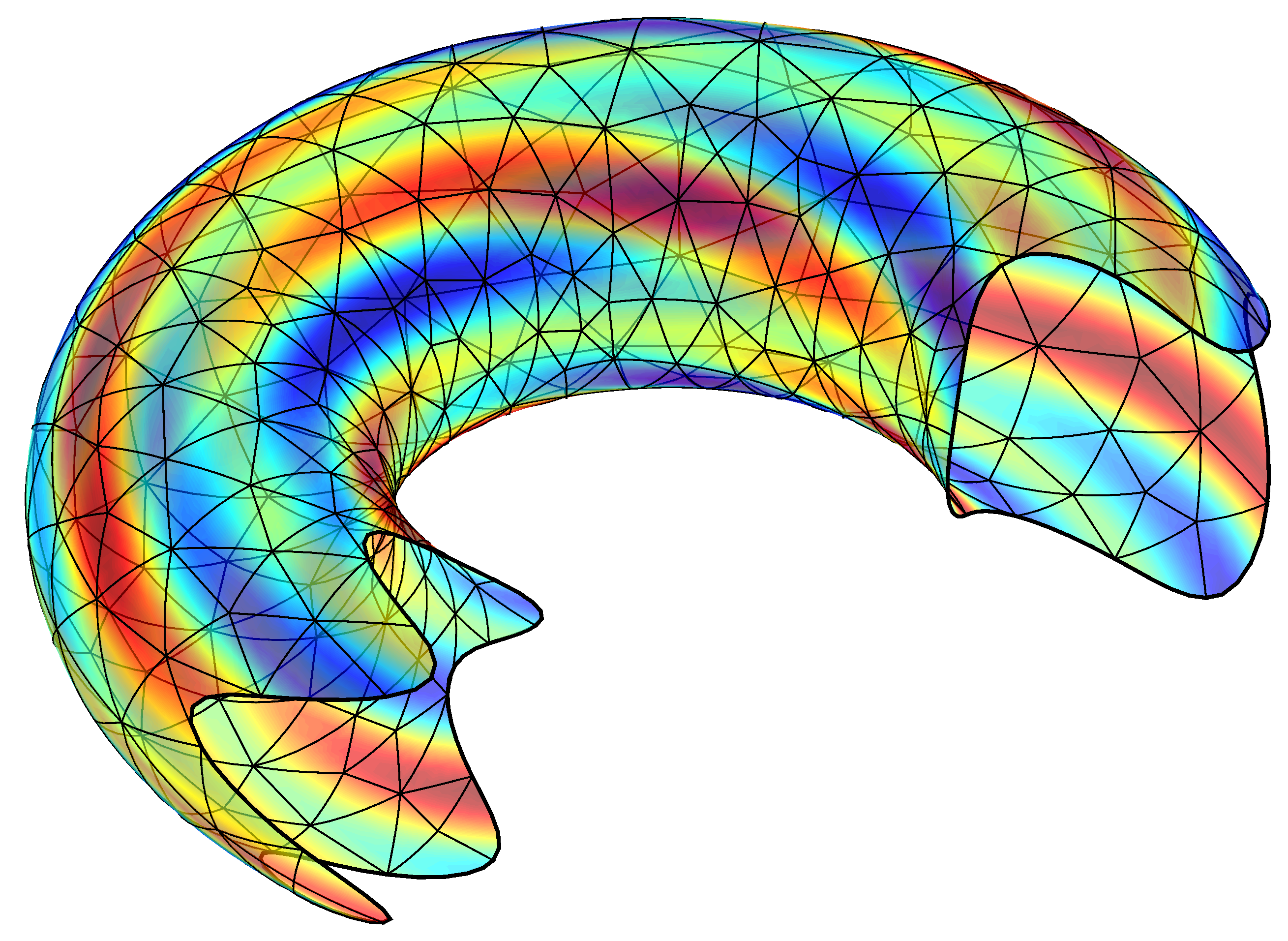}
\caption{Numerical solution for the model problem using $k=3$ and meshsize $h=1/4$.
}
\label{fig:solution}
\end{figure}

\begin{figure}
\centering
\includegraphics[width=0.70\linewidth]{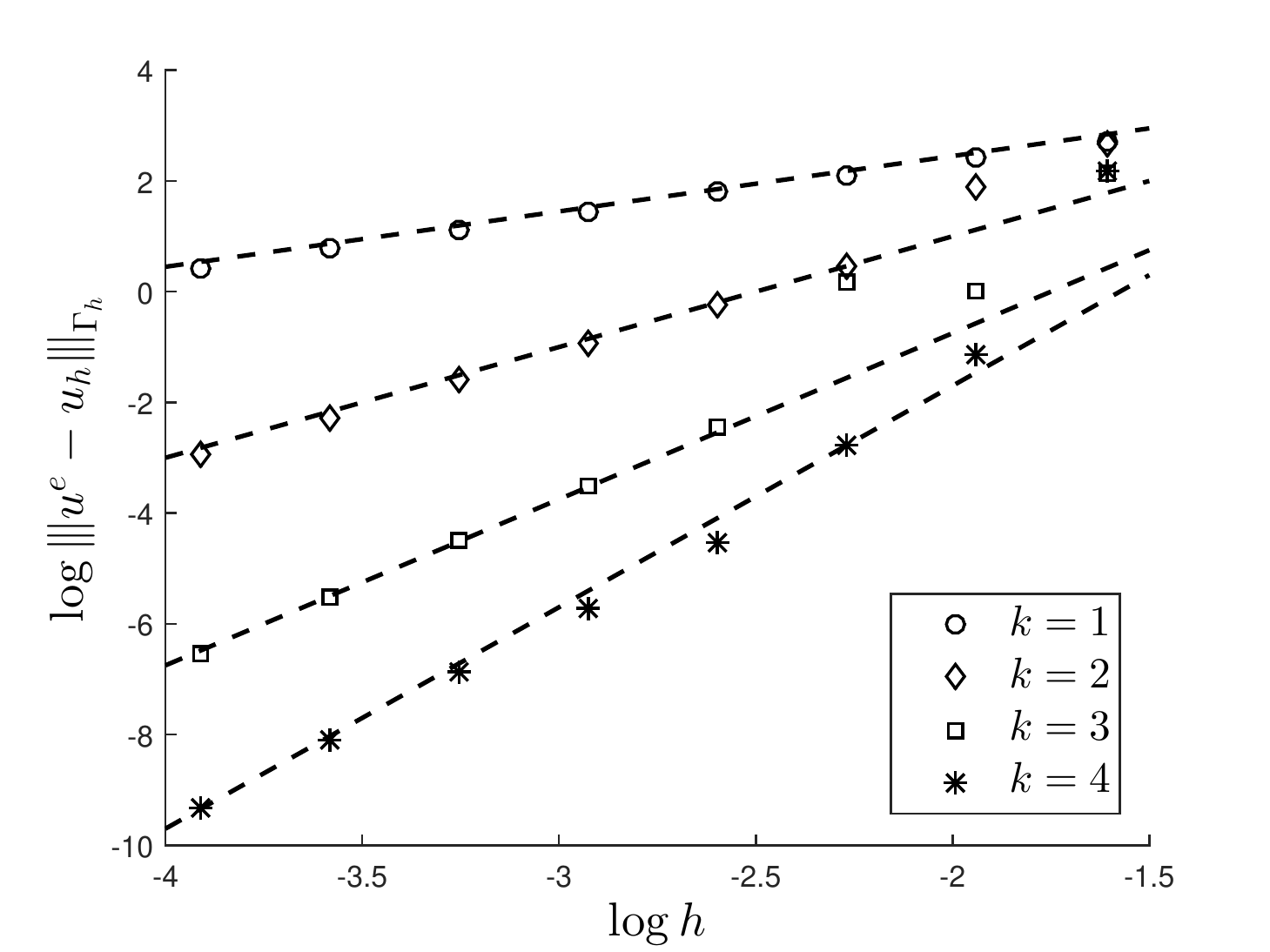}
\caption{Convergence study of the model problem in energy norm with reference lines proportional to $h^{k}$. Note the instability in convergence rate for coarse meshes and higher order geometry approximation.
}
\label{fig:convenergy}
\end{figure}

\begin{figure}
\centering
\includegraphics[width=0.70\linewidth]{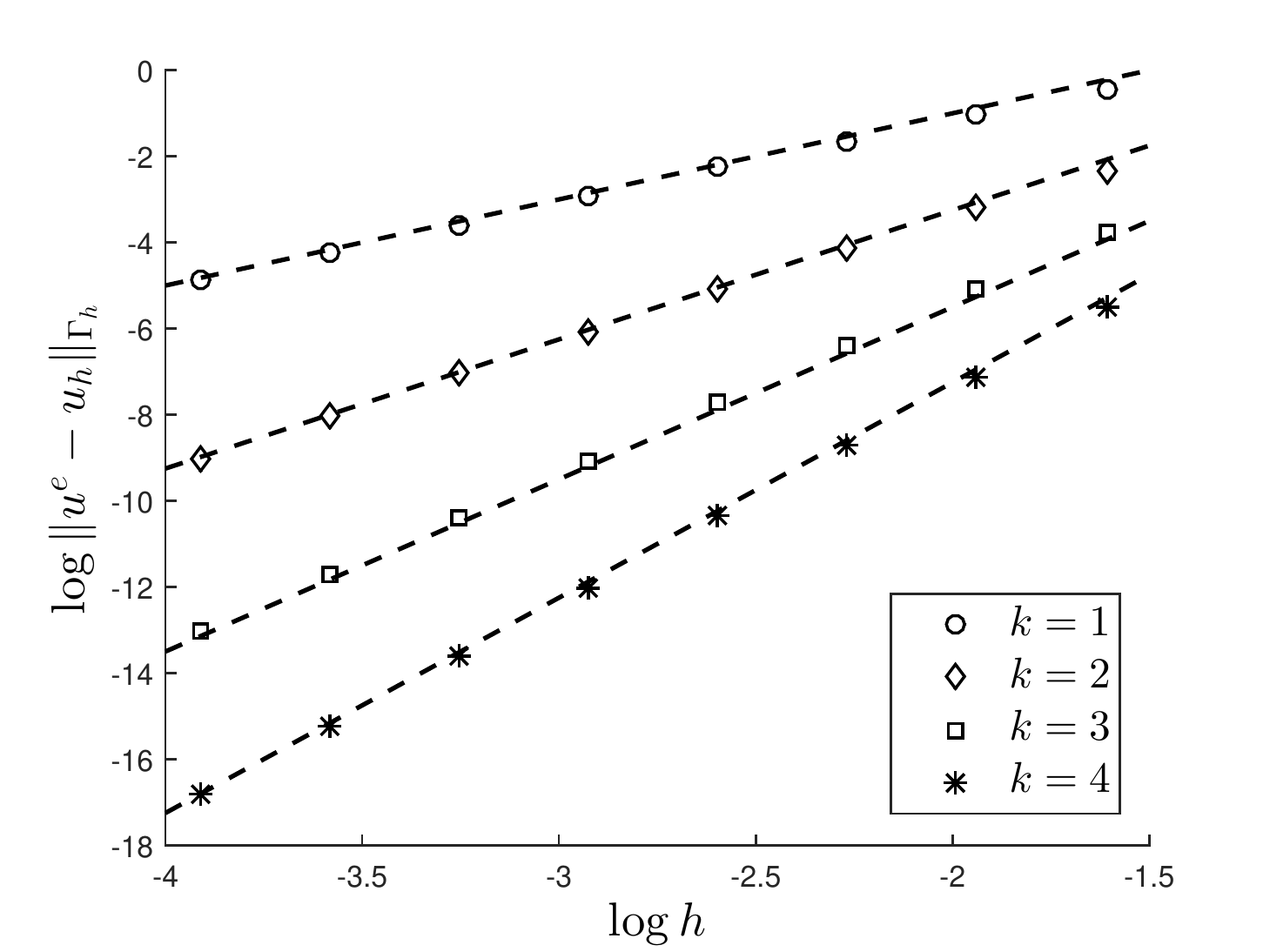}
\caption{Convergence study of the model problem in $L^2$ norm with reference lines proportional to $h^{k+1}$.
}
\label{fig:convL2}
\end{figure}

\begin{figure}
\centering
\includegraphics[width=0.60\linewidth]{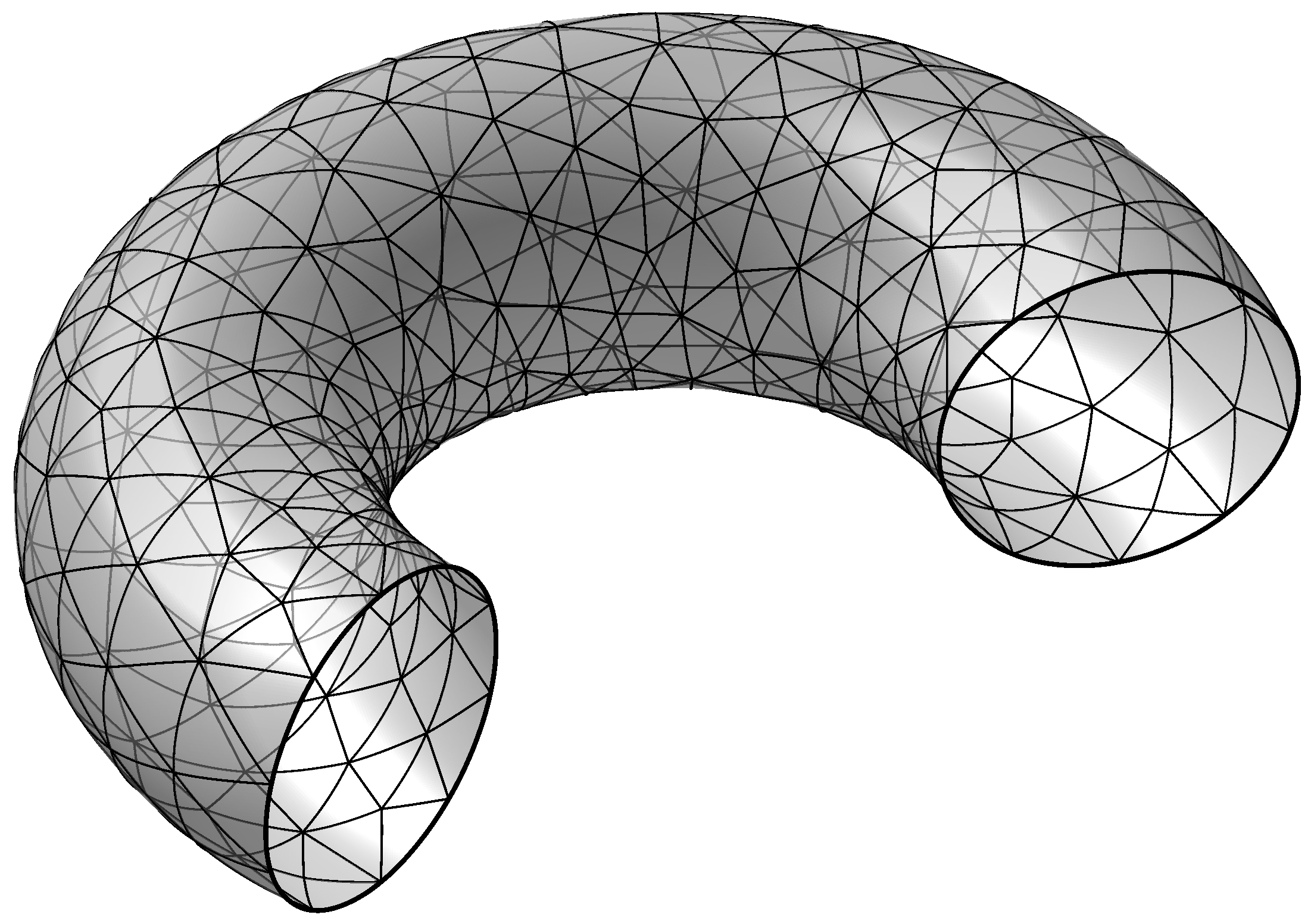}
\caption{Mesh for a simplified version of the model problem ($N_1=N_2=0$) using geometric interpolation order $k=3$ and meshsize $h=1/4$. For this problem $\Gammahl=\Gamma$, i.e. there is no boundary correction.
}
\label{fig:mesh_simplified}
\end{figure}

\begin{figure}
\centering
\includegraphics[width=0.70\linewidth]{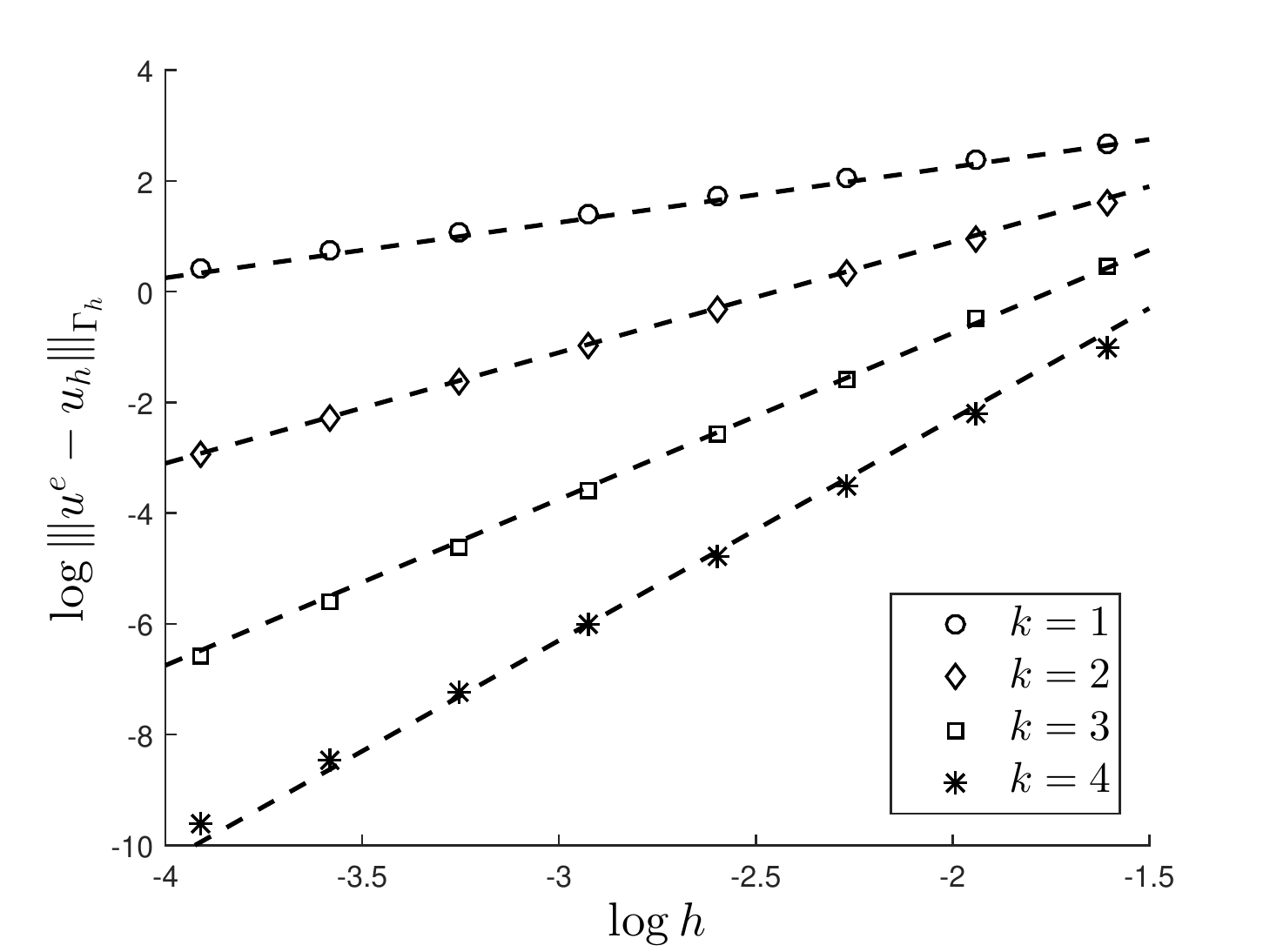}
\caption{Convergence study for a simplified version of the model problem ($N_1=N_2=0$) in energy norm with reference lines proportional to $h^{k}$. Note that there is no instability in convergence rate for coarse meshes.
}
\label{fig:convenergy_simplified}
\end{figure}


\paragraph{Acknowledgement.} This research was supported in part by EPSRC, UK, Grant No. EP/J002313/1, the Swedish Foundation for Strategic Research Grant No.\ AM13-0029, the Swedish Research Council Grants Nos.\ 2011-4992, 
2013-4708, and Swedish strategic research programme eSSENCE.

\bibliographystyle{plain}
\bibliography{BDT}
\end{proof}

%
%

\end{document}